\documentclass[12pt,a4paper]{amsart}
\usepackage[top=35mm, bottom=32mm, left=30mm, right=30mm]{geometry}
\usepackage{newtxtext,newtxmath}
\usepackage{mathrsfs}

\usepackage{url}
\usepackage[all]{xy}
\usepackage[colorlinks=true,citecolor=blue]{hyperref}

\theoremstyle{plain}

\newtheorem{thm}{Theorem}[section]
\newtheorem{lem}[thm]{Lemma}
\newtheorem{prop}[thm]{Proposition}
\newtheorem{cor}[thm]{Corollary}

\theoremstyle{definition}

\newtheorem{rem}[thm]{Remark}
\newtheorem{ques}{Question}

\newcommand{\Z}{{\mathbb{Z}_+}}
\newcommand{\N}{\mathbb{N}}

\newcommand{\ep}{\varepsilon}
\newcommand{\ra}{\rightarrow}

\newcommand{\T}{\mathbb{T}}
\newcommand{\R}{\mathbb{R}}

\DeclareMathOperator{\spann}{span}

\newcommand{\dd}{\mathop{}\!\mathrm{d}}

\newcommand{\I} {\mathcal I}

\begin{document}
\title{Bounded complexity, mean equicontinuity and discrete spectrum}
\author[W. Huang, J. Li, J. Thouvenot, L. Xu and X. Ye]
{Wen Huang, Jian Li, Jean-Paul Thouvenot, Leiye Xu and Xiangdong Ye}

\date{\today}

\address[W.~Huang, L.~Xu and X.~Ye]
{Wu Wen-Tsun Key Laboratory of Mathematics, USTC, Chinese Academy of Sciences and Department of Mathematics, University of Science and
Technology of China, Hefei, Anhui 230026, China}
\email{wenh@mail.ustc.edu.cn, leoasa@mail.ustc.edu.cn, yexd@ustc.edu.cn}

\address[J.~Li]{Department of Mathematics, Shantou University,
Shantou, Guangdong 515063, P.R. China}
\email{lijian09@mail.ustc.edu.cn}

\address[J.~Thouvenot]{Universit\'e Paris 6-LPMA, Case courrier 188, 4 Place Jussieu, 75252 Paris Cedex 05, France}
\email{jeanpaul.thouvenot@gmail.com}



\begin{abstract}
We study dynamical systems which have bounded complexity with respect to
three kinds metrics: the Bowen metric $d_n$, the max-mean metric $\hat{d}_n$
and the mean metric $\bar{d}_n$, both in topological dynamics and ergodic theory.

It is shown that a topological dynamical system $(X,T)$ has bounded complexity with respect to $d_n$
(resp.\ $\hat{d}_n$) if and only if
it is equicontinuous (resp.\ equicontinuous in the mean).
However, we construct minimal systems which have  bounded complexity with respect to
$\bar{d}_n$ but not equicontinuous in the mean.

It turns out that an invariant  measure $\mu$ on $(X,T)$ has bounded complexity with respect to $d_n$
if and only if $(X,T)$ is $\mu$-equicontinuous. Meanwhile, it is shown that
$\mu$ has bounded complexity with respect to $\hat{d}_n$
if and only if  $\mu$ has bounded complexity with respect to $\bar{d}_n$
if and only if $(X,T)$ is $\mu$-mean equicontinuous if and only if it has
discrete spectrum.
\end{abstract}

\maketitle
\section{Introduction}
Throughout this paper,
by a topological dynamical system (t.d.s.\ for short)
we mean a pair $(X,T)$,
where $X$ is a compact metric space with a metric $d$
and $T$ is a continuous map from $X$ to itself.
Let $\mathcal{B}_X$ be the Borel $\sigma$-algebra on $X$
and $\mu$ be a probability measure on $(X,\mathcal{B}_X)$.
We say that $\mu$ is an invariant measure for $T$
if for every $B\in\mathcal{B}_X$, $\mu(T^{-1}B)=\mu(B)$.

Entropy is a very useful invariant to describe the complexity of a dynamical system
which measures the rate of the exponential growth of the orbits. For some simple systems
(for example dynamical systems with zero entropy) it is useful to consider the complexity function
itself. This kind of considerations can be traced back to the work by Morse and Hedlund, who studied the
complexity function of a subshift and proved that the boundedness of the function is equivalent to
the eventual periodicity of the system (for progress on the high dimensional analogue see \cite{CK}). In \cite{Fer}, Ferenczi studied  measure-theoretic complexity of
ergodic systems using $\alpha$-names of a partition and the Hamming distance. He proved that when the
measure is ergodic, the complexity function is bounded
if and only if the system has discrete spectrum (for the result dealing with non-ergodic case, see \cite{YU}). In \cite{Katok} Katok introduced a notion using the modified notion of spanning sets
with respect to an invariant measure $\mu$ and an error $\ep$, which can be used to define the complexity function.
In \cite{BHM}, Blanchard, Host and Maass studied topological complexity via the complexity function of an open cover and
showed that the complexity function is bounded for any open cover if and only if the system is equicontinuous.

Recently, in the investigation of the Sarnak's conjecture, Huang, Wang and Ye \cite{HWY} introduced the measure complexity of
an invariant  measure $\mu$ similar to the one introduced by Katok \cite{Katok}, by using the mean metric instead of the
Bowen metric (for discussion and results related to mean metric, see also \cite{LM,AR}).
They showed that
if an invariant measure has discrete spectrum, then the measure complexity with respect to this invariant measure is bounded. An
open question was posed as whether the converse statement holds. Motivated by this open question and
inspired by the discussions in \cite{Fer, Katok, BHM, HLY, LTY, F-1, Felip, HKM-2}, in this paper,
we study topological and measure-theoretic complexity via a sequence of metrics induced by a metric $d$,
namely the metrics $d_n,\hat d_n$ and $\bar d_n$.

\medskip
To be precise, for $n\in\mathbb{N}$, we define three  metrics on $X$ as follows. For $x,y\in X$, let
\[d_n(x,y)=\max\{d(T^ix,T^iy)\colon 0\le i\le n-1\},\    \bar{d}_n(x,y)= \frac{1}{n}\sum_{i=0}^{n-1}d(T^ix,T^iy) \] and
\[\hat{d}_n(x,y)= \max\Bigl\{\bar d_k(x,y)\colon 1\le k\le n\Bigr\}.\]

It is clear that for all $x,y\in X$,
\[d_n(x,y)\geq \hat{d}_n(x,y) \geq \bar{d}_n(x,y).\]

For $x\in X$, $\ep>0$ and a metric $\rho$ on $X$, let $B_{\rho}(x,\ep)=\{y\in X\colon \rho(x,y)<\ep\}.$
We say that  a dynamical system $(X,T)$ has
bounded topological complexity
with respect to a sequence of metrics $\{\rho_n\}$
if for every $\ep>0$ there exists a positive integer $C$
such that for each $n\in\N$ there are points
$x_1,x_2,\dotsc,x_m\in X$ with $m\le C$ satisfying
$X=\bigcup_{i=1}^m B_{\rho_n}(x_i,\ep)$.
In this paper  we will focus on the situation when $\rho_n=d_n$, $\hat d_n$ and $\bar d_n$.

We also study the measure-theoretic complexity of invariant measures. That is, for a given $\ep>0$
and an invariant measure $\mu$  we consider the measure complexity with respect to $\{\rho_n\}$ with $\rho_n=d_n,\hat d_n$ and $\bar d_n$, defined by
$$\min\{m\in\Z: \exists x_1,\ldots,x_m\in X, \mu(\cup_{i=1}^m B_{\rho_n}(x_i,\ep))>1-\ep\}.$$
As expected, the bounded complexity
of a topological dynamical system or a measure preserving system is related to various notions of equicontinuity.

It is shown that (see Theorem \ref{thm:T-equicontinuous-subset} and Theorem \ref{thm:strongly-mean-equi}) a topological dynamical system $(X,T)$ has bounded complexity with respect to $d_n$
(resp.\ $\hat{d}_n$) if and only if
it is equicontinuous (resp.\ equicontinuous in the mean).
At the same time, we construct minimal systems which have  bounded complexity with respect to
$\bar{d}_n$ but not equicontinuous in the mean, which are not uniquely ergodic or uniquely ergodic (see Proposition \ref{prop:key-example1} and  Proposition \ref{prop:key-example2}).

It turns out that an invariant  measure $\mu$ on $(X,T)$ has bounded complexity with respect to $d_n$
if and only if $(X,T)$ is $\mu$-equicontinuous (see Theorem \ref{thm-mu-equi}). Meanwhile, it is shown that
$\mu$ has bounded complexity with respect to $\hat{d}_n$
if and only if  $\mu$ has bounded complexity with respect to $\bar{d}_n$
if and only if $(X,T)$ is $\mu$-mean equicontinuous if and only if $(X,T)$ is $\mu$-equicontinuous in the mean
if and only if it has
discrete spectrum (see Theorem  \ref{thm:mu-eq-in-mean}, Theorem \ref{nov-5} and Theorem \ref{thm:mu-discrete-spectrum}).

\medskip
The structure of the paper is the following. In Section 2, we recall some basic notions which we will use in the paper. In Section 3,
we prove the topological results for systems with bounded complexity with respect to three kinds of metrics. In Section 4,
we consider the corresponding results in the measure-theoretical setting. In the Appendix we give some examples.

\medskip

\noindent {\bf Acknowledgments.}
We would like to thank Bryna Kra for the helpful discussion on the subject,
specially on Theorem \ref{thm:mu-discrete-spectrum} and allowing us to include an example (due to Cyr and Kra) in Appendix B.
We also would like to thank Nhan-Phu Chung for bringing \cite{VZP12} into our attention.
W. Huang was partially supported by NNSF of China (11431012,11731003),
J. Li (Corresponding author) was partially supported by NNSF of China (11771264) and NSF of Guangdong Province (2018B030306024),
L. Xu was partially supported by NNSF of China (11801538, 11871188)
and X. Ye was partially supported by NNSF of China (11431012)

\section{Preliminaries}

In this section we recall some notions and aspects of dynamical systems which will be used later.
\subsection{General notions}
In the article, the sets of integers, nonnegative integers
and natural numbers are denoted by $\mathbb{Z}$, $\mathbb{Z}_+$ and $\mathbb{N}$, respectively.
We use $\#(A)$ to denote the number of elements of a finite set $A$.

A t.d.s. $(X, T)$ is \emph{transitive} if for each pair of non-empty
open subsets $U$ and $V$, $N(U, V)=\{n \in \mathbb{Z}_+:
U\cap T^{-n}V \neq \emptyset\}$ is infinite; it is \emph{totally
transitive} if $(X, T^n)$ is transitive for each $n \in
\mathbb{N}$; and it is \emph{weakly mixing} if $(X \times X, T \times
T)$ is transitive. We say that $x \in X$ is a \emph{transitive point}
if its orbit $\mathrm{Orb}(x, T)=\{x, Tx, T^2x, \ldots\}$ is dense
in $X$. The set of transitive points is denoted by
$\mathrm{Trans}(X,T)$. It is well known that if $(X, T)$ is transitive,
then $\mathrm{Trans}(X,T)$ is a dense $G_{\delta}$ subset of $X$.

A t.d.s. $(X, T)$ is \emph{minimal} if $\mathrm{Trans}(X,T)=X$, i.e., it
contains no proper subsystems. A point $x \in X$ is called a \emph{minimal point} or \emph{almost periodic point} if
$(\overline{\mathrm{Orb}(x, T)}, T)$ is a minimal subsystem of $(X,
T)$.

\subsection{Equicontinuity and mean equicontinuity}
A t.d.s.\ $(X,T)$ is called \emph{equicontinuous} if for every $\ep>0$ there is a $\delta>0$ such that
whenever $x,y\in X$ with $d(x,y)<\delta$, $d(T^nx,T^ny)<\ep$ for $n=0,1,2,\dotsc$.
It is well known that a t.d.s.\ $(X,T)$ with $T$ being surjective is equicontinuous if and only if
there exists a compatible metric $\rho$ on $X$ such that $T$ acts on $X$ as an isometry, i.e.,
$\rho(Tx,Ty)=\rho(x,y)$ for any $x,y\in X$.
Moreover, a transitive equicontinuous system is conjugate to
a minimal rotation on a compact abelian metric group,
and $(X,T,\mu)$ has discrete spectrum, where $\mu$ is the unique normalized Haar measure on $X$.

When studying dynamical systems with discrete spectrum, Fomin~\cite{F51} introduced a notion
called \emph{stable in the mean in the sense of Lyapunov} or simply \emph{mean-L-stable}.
A t.d.s.\ $(X,T)$ is \emph{mean-L-stable}
if for every $\ep>0$, there is a $\delta>0$ such that $d(x,y)<\delta$
implies $d(T^nx,T^ny)<\ep$ for all $n\in\Z$ except a set of
upper density less than $\ep$.
Fomin proved that if a minimal system is mean-L-stable then it is uniquely ergodic.
Mean-L-stable systems are also discussed briefly by Oxtoby in~\cite{O52},
and he proved that each transitive mean-L-stable system is uniquely ergodic.
Auslander in~\cite{A59} systematically studied mean-L-stable systems, and provided new examples.
See Scarpellini \cite{S82} for a related work.
It was an open question whether every ergodic invariant measure on a mean-L-stable system has discrete spectrum~\cite{S82}.
This question was answered affirmatively by Li, Tu and Ye in \cite{LTY}.

A t.d.s.\ $(X,T)$ is called \emph{mean equicontinuous} (resp. \emph{equicontinuous in the mean}) if for every $\ep>0$,
there exists a $\delta>0$ such that whenever $x,y\in X$ with $d(x,y)<\delta$,
$\limsup_{n\to\infty} \bar d_n(x,y)<\ep$ (resp. $\bar d_n(x,y)<\ep$ for each $n\in\N$).
It is not hard to show that a dynamical system is mean equicontinuous if and only if it is mean-L-stable.
For works related to mean equicontinuity, we refer to \cite{LTY, DG16, Felip, J-G, Li}.
We remark that by the result in \cite{DG16}, a minimal null or tame system is mean equicontinuous.
We will show in this paper that a minimal system is mean equicontinuous if and only if it is equicontinuous in the mean
(for the proof for the general case, see \cite{QZ18}).

\subsection{$\mu$-equicontinuity and $\mu$-mean equicontinuity}
When studying the chaotic behaviors of dynamical systems, Huang, Lu and Ye \cite{HLY} introduced a notion which connects the equicontinuity
with respect to a subset or a measure.

Following \cite{HLY}, for a t.d.s.\ $(X,T)$, we say that a subset $K$ of $X$ is \emph{equicontinuous} if for every $\ep>0$,
there exists a $\delta>0$ such that $d(T^nx,T^ny)<\ep$ for all $n\in \Z$ and
all $x,y\in K$ with $d(x,y)<\delta$. For an invariant measure $\mu$ on $(X,T)$,
we say that $T$ is \emph{$\mu$-equicontinuous} if for any $\tau>0$
there exists a $T$-equicontinuous measurable subset $K$ of $X$ with $\mu(K)>1-\tau$.
It was shown in \cite{HLY} that if $(X,T)$ is $\mu$-equicontinuous and $\mu$ is ergodic
then $\mu$ has discrete spectrum. We note that $\mu$-equicontinuity was studied further in \cite{F-1}.

In the process to study mean equicontinuity, the above notions were generalized to
mean equicontinuity with respect to  an invariant measure by Garc{\'\i}a-Ramos in \cite{Felip}.
Particularly, he proved that for an ergodic invariant measure $\mu$,
$(X,T)$ is $\mu$-mean equicontinuous
if and only if $\mu$ has discrete spectrum.
For a different approach, see \cite{L16}.

\subsection{Hausdorff metric}
Let $K(X)$ be the hyperspace on $X$, i.e., the space of non-empty closed subsets of $X$ equipped with
the Hausdorff metric $d_H$ defined by
\[d_H(A,B)=\max\Bigl\{\max_{x\in A}\min_{y\in B} d(x,y),\ \max_{y\in B}\min_{x\in A}
d(x,y)\Bigr\}\ \text{for}\ A,B\in K(X).\]
As $(X,d)$ is compact, $(K(X),d_H)$ is also compact.
For $n\in\N$, it is easy to see that the map $X^n\to K(X)$,
$(x_1,\dotsc,x_n)\mapsto \{x_1,\dotsc,x_n\}$,
is continuous. Then $\{A\in K(X)\colon \#(A)\leq n\}$ is a closed subset of $K(X)$.

\subsection{Discrete spectrum}
Let $(X, T )$ be an invertible t.d.s.,
that is, $T$ is a homeomorphism on $X$.
Let $\mu$ be an invariant  measure on $(X,T)$ and
let $L^2(\mu) = L^2(X,\mathcal{B}_X,\mu)$ for short.
An \emph{eigenfunction } for $\mu$ is some non-zero
function $f\in  L^2(\mu)$ such that $Uf := f\circ T =\lambda f$
for some $\lambda \in \mathbb{C}$.
In this case, $\lambda$ is
called the \emph{eigenvalue} corresponding to $f$. It is easy to see every eigenvalue has norm one,
that is $|\lambda| = 1$. If $f\in L^2(\mu)$ is an eigenfunction, then $\{U^nf : n\in \mathbb{Z}\}$ is  precompact in $L^2(\mu)$,
that is the closure of $\{U^nf : n\in \mathbb{Z}\}$ is compact in $L^2(\mu)$.
Generally, we say that $f$ is \emph{almost periodic} if $\{U^nf : n\in \mathbb{Z}\}$ is precompact
in $L^2(\mu)$. It is well known that the set of all bounded almost periodic functions forms a
$U$-invariant and conjugation-invariant subalgebra of $L^2(\mu)$ (denoted by $A_c$). The set of all
almost periodic functions is just the closure of $A_c$ (denoted by $H_c$), and is also spanned by
the set of eigenfunctions.
The invariant measure $\mu$ is said  to have \emph{discrete spectrum} if $L^2(\mu)$
is spanned by the set of eigenfunctions, that is $H_c = L^2(\mu)$.
We remark that when $\mu$ is not
ergodic, the structure of a system $(X,T,\mu)$ with discrete spectrum can be very complicated, we refer to \cite{Kwi8, EN} and
the example we provide at the end of Section 4 for details.

\section{Topological dynamical systems with bounded topological complexity}
In this section we will study the topological complexity of dynamical systems
with respect to three kinds of metrics.

\subsection{Topological complexity with respect to  $\{d_n\}$}
Let $(X,T)$ be a t.d.s.
For $n\in\mathbb{N}$ and $x,y\in X$, define
\[d_n(x,y)=\max\{d(T^ix,T^iy)\colon i=0,1,\dotsc,n-1\}.\]
It is easy to see that for each $n\in\N$, $d_n$ is a metric on $X$ which is topologically equivalent to
the metric $d$.
Let $x\in X$ and $\ep>0$.
The open ball of centre $x$ and radius $\ep$ in the metric $d_n$ is
\[B_{d_n}(x,\ep)=\{y\in X\colon d_n(x,y)<\ep\}=\bigcap_{i=0}^{n-1}T^{-i}B(T^ix,\ep).\]
Let $K$ be a subset of $X$, $n\in\N$ and $\ep>0$.
A subset $F$ of $K$ is said to \emph{$(n,\ep)$-span $K$ with respect to $T$} if
for every $x\in K$ there exists $y\in F$ with $d_n(x,y)<\ep$, that is
\[K\subset \bigcup_{x\in F}B_{d_n}(x,\ep).\]
Let $\spann_K(n,\varepsilon)$ denote the smallest cardinality of any
$(n,\ep)$-spanning set for $K$ with respect to $K$, that is
\[\spann_K(n,\varepsilon)=
\min\biggl\{\#(F)\colon
F\subset K\subset \bigcup_{x\in F}B_{d_n}(x,\ep)\biggr\}.\]
We say that a subset $K$ of $X$ has \emph{bounded topological complexity
with respect to $\{d_n\}$}
if for every $\ep>0$ there exists a positive integer $C=C(\ep)$
such that $\spann_K(n,\varepsilon)\leq C$ for all $n\geq 1$.
If the whole set $X$ has bounded topological complexity
with respect to $\{d_n\}$,
we will say that the dynamical system $(X,T)$ has
the property.

We first show that a subset with bounded topological complexity
with respect to $\{d_n\}$
is equivalent to the equicontinuity property.

\begin{thm}\label{thm:T-equicontinuous-subset}
Let $(X,T)$ be a t.d.s.\ and $K\subset X$ be a compact set.
Then $K$ has bounded topological complexity
with respect to $\{d_n\}$ if and only if it is equicontinuous.
\end{thm}

\begin{proof}
($\Leftarrow$)
Fix $\varepsilon>0$.
By the definition of equicontinuity,
there exists $\delta>0$ such that $d(T^nx,T^ny)<\ep$ for all $n\in \Z$ and
all $x,y\in K$ with $d(x,y)<\delta$.
By the compactness of $K$, there exists a finite subset $F$ of $K$
such that $K\subset \bigcup_{x\in F} B(x,\delta)$.
Then $K\subset \bigcup_{x\in F} B_{d_n}(x,\ep)$ for all $n\geq 1$.
So $K$ has bounded topological complexity with respect to $\{d_n\}$.

\medskip

($\Rightarrow$)
Assume the contrary that $K$ is not equicontinuous.
There exists $\ep>0$
such that for any $k\geq 1$ there are $x_k,y_k\in K$
and $m_k\in\N$ such that $d(x_k,y_k)<\frac{1}{k}$ and $d(T^{m_k}x_k,T^{m_k}y_k)\geq \ep$.
Without loss of generality, we may assume that $x_k\to x_0$ as
$k\to\infty$.
Then we have $x_0\in K$ and $y_k\to x_0$ as $k\to\infty$.
For any $k\in\N$, by the triangle inequality, either
$d(T^{m_k}x_k,T^{m_k}x_0)\geq \frac{\ep}{2}$ or $d(T^{m_k}y_k,T^{m_k}x_0)\geq \frac{\ep}{2}$.
Without loss of generality, we always have $d(T^{m_k}x_k,T^{m_k}x_0)\geq \frac{\ep}{2}$ for all $k\in \N$.
Then $d_{m_k+1}(x_0,x_k)\geq \ep/2$ for all $k\in\N$.

As $K$ has bounded topological complexity with respect to $\{d_n\}$,
for the constant $\varepsilon/6$,
there exists $C>0$ such that for every $n\geq 1$
there exists a subset $F_n$ of $K$ with $\#(F_n)\leq C$
such that $K\subset \bigcup_{x\in F_n} B_{d_n}(x,\ep/6)$.
We view $\{F_n\}$ as a sequence in the hyperspace $K(X)$.
By the compactness of $K(X)$, there is a subsequence
$F_{n_i}\to F$ as $i\to\infty$ in the Hausdorff metric $d_H$.
As $F_n\subset K$ and $K$ is compact, we have $F\subset K$.
By the fact $\{A\in K(X)\colon  \#(A)\leq C\}$ is closed,
we have $\#(F)\leq C$.
For any $i\in\N$ and any $x\in K$,
there exists $z_{n_i}\in F_{n_i}$ such that $d_{n_i}(x,z_{n_i})<\ep/6$.
Without loss of generality, assume that $z_{n_i}\to z$
as $i\to\infty$.
Then $z\in F$.
As the sequence $\{d_n\}$ of metrics is increasing,
that is $d_{n}(u,v)\leq d_{n+1}(u,v)$
for all $u,v\in X$ and $n\in\N$,
we have  $d_{n_i}(x,z_{n_j})\le d_{n_j}(x,z_{n_j})<\ep$ for all $j\geq i$.
Letting $j$ go to infinity, we get $d_{n_i}(x,z)\leq \ep/6$.
This implies that
\[
K\subset \bigcup_{z\in F} \{x\in K \colon d_{n_i}(x,z)\leq \ep/6\}
\] for all $n_i$.
By the monotonicity of $\{d_n\}$, we have
\[
K\subset \bigcup_{z\in F} \{x\in K \colon d_{n}(x,z)\leq \ep/6\}
\] for all $n\in\N$.
Enumerate $F$ as $\{z_1,\dotsc,z_m\}$ and let
\[K_j=\bigcap_{n=1}^\infty\{x\in K\colon d_n(x,z_j)\leq \ep/6 \}\]
for $j=1,\dotsc,m$.
Then each $K_i$ is a closed set.
By the  monotonicity of $\{d_n\}$, we have
$K=\bigcup_{j=1}^m K_j$.

For the sequence $\{x_k\}$ in $K$,
passing to a subsequence if necessary
we assume that the sequence $\{x_k\}$ is in the same $K_{j}$.
As $K_{j}$ is closed, $x_0$ is also in $K_{j}$.
Note that for any $u,v\in K_{j}$ and any $n\geq 1$,
$d_n(u,v)\leq d_n(u,z_{j})+d_n(z_{j},v)\leq \ep/3$.
Particularly, we have $d_{m_k+1}(x_0,x_k)\leq \ep/3$
for any $k\in\N$, which is a contradiction.
\end{proof}

\begin{rem}
In the definition of  $(n,\varepsilon)$-spanning set $F$ of $K$,
we require $F$ to be a subset of $K$.
In fact we can define
\[\spann'_K(n,\varepsilon)=\min\{\#(F)\colon F\subset X
\text{ and }K\subset \bigcup_{x\in F}B_{d_n}(x,\ep)\}.\]
It is clear that $\spann_K(n,2\varepsilon)
\leq \spann'_K(n,\varepsilon)\leq \spann_K(n,\varepsilon)$.
So Proposition~\ref{thm:T-equicontinuous-subset} still holds
if in the definition of topological complexity
with respect to $\{d_n\}$
we replace $\spann_K(n,\varepsilon)$ by $\spann'_K(n,\varepsilon)$.
\end{rem}

\begin{cor}\label{cor:equicont-bounded}
A dynamical system $(X,T)$ is equicontinuous if and only if
for every $\varepsilon>0$
there exists  a positive integer $C$
such that $\spann_X(n,\varepsilon)\leq C$ for all $n\geq 1$.
\end{cor}

\begin{rem} \label{rem:open-cover}
It is shown in \cite{BHM} that
the complexity defined by using the
open covers is bounded if and only if the system is equicontinuous.
In fact, we can prove Corollary \ref{cor:equicont-bounded}
by using this result and the
the fact that \cite[Theorem 7.7]{Wal}
if $\alpha$ is an open cover of $X$ with Lebesgue number $\delta$ then
$$N(\vee_{i=0}^{n-1}T^{-i}\alpha)\le \spann_X(n,\delta/2).$$
\end{rem}

\subsection{Topological complexity with respect to  $\{\hat{d}_n\}$}
For $n\in\mathbb{N}$ and $x,y\in X$,  define
\[\hat{d}_n(x,y)= \max\Bigl\{\frac{1}{k}\sum_{i=0}^{k-1}d(T^ix,T^iy)\colon k=1,2,\dotsc,n\Bigr\}.\]
It is easy to see that for each $n\in\N$, $\hat{d}_n$
is a metric on $X$ which is topologically equivalent to
the metric $d$.
For $x\in X$ and $\ep>0$, let
$B_{\hat{d}_n}(x,\ep)=\{y\in X\colon \hat{d}_n(x,y)<\ep\}.$
Let $K$ be a subset of $X$.
For $n\in\N$ and $\ep>0$, define
\[\widehat{\spann}_K(n,\varepsilon)=
\min\biggl\{\#(F)\colon
F\subset K\subset \bigcup_{x\in F}B_{\hat{d}_n}(x,\ep)\biggr\}.\]
We say that a subset $K$ of $X$ has \emph{bounded topological complexity
with respect to $\{\hat{d}_n\}$}
if for every $\ep>0$ there exists a positive integer $C=C(\ep)$
such that $\widehat{\spann}_K(n,\varepsilon)\leq C$
for all $n\geq 1$.

As $ \hat{d}_n(x,y)\leq d_n(x,y)$ for all $n\in\N$ and $x,y\in X$,
if $K$ has bounded topological complexity
with respect to $\{d_n\}$ then it is also bounded topological complexity
with respect to $\{\hat{d}_n\}$.
We say that a subset $K$ of $X$ is \emph{equicontinuous in the mean}
if for every $\ep>0$,
there exists a $\delta>0$ such that $\hat{d}_n(x,y)<\ep$ for all $n\in \Z$ and
all $x,y\in K$ with $d(x,y)<\delta$.

The following result follows the same lines in
Theorem~\ref{thm:T-equicontinuous-subset},
just replace the distance $d_n$
by $\hat{d}_n$, as the sequence $\{\hat{d}_n\}$ of metrics is also increasing.

\begin{thm}\label{thm:strongly-mean-equi}
Let $(X,T)$ be a t.d.s.\ and $K$ be a compact subset of $X$.
Then $K$ has bounded topological complexity
with respect to  $\hat{d}_n$ if and only if it is equicontinuous in the mean.
\end{thm}

We say that a subset $K$ of $X$
is \emph{mean equicontinuous}
if for every $\ep>0$,
there exists a $\delta>0$ such that
\[
\limsup_{n\to\infty}\frac{1}{n}\sum_{i=0}^{n-1}d(T^ix,T^iy)<\ep
\]
for all $x,y\in K$ with $d(x,y)<\delta$.
If $X$ is mean equicontinuous then we say that
$(X,T)$ is \emph{mean equicontinuous}.
It is clear that if $K$ is equicontinuous in the mean then it is mean equicontinuous.
We can show that for minimal systems they are equivalent.
\begin{prop}
	Let $(X,T)$ be  a minimal t.d.s.
	Then $(X,T)$ is mean equicontinuous
	if and only if  equicontinuous in the mean.
\end{prop}
\begin{proof}
	It is clear that equicontinuity in the mean implies mean equicontinuity.
	
	Assume that $(X,T)$ is mean equicontinuous.
	For each $\ep>0$ there is $\delta_1>0$ such that
	if $d(x,y)<\delta_1$
	then \[\limsup_{n\to\infty}\frac{1}{n}\sum_{i=0}^{n-1}d(T^ix,T^iy)<\frac{\ep}{8}.\]
	Fix $z\in X$. For each $N\in\N$, let
	\[A_N=\biggl\{x\in \overline{B(z,\delta_1/2)} \colon
	\frac{1}{n}\sum_{i=0}^{n-1}d(T^ix,T^iz) \leq \frac{\ep}{4},\
	n= N,N+1,\dotsc \biggr\}.\]
	Then $A_N$ is closed and
	$\overline{B(z,\delta_1/2)}=\bigcup_{N=1}^\infty A_N$.
	By the Baire Category Theorem,
	there is $N_1\in\N$ such that $A_{N_1}$
	contains an open subset $U$ of $X$.
	By the minimality we know that there is $N_2\in\N$ with $\bigcup_{i=0}^{N_2-1} T^{-i}U=X$.
	Let $\delta_2$ be the Lebesgue number of the open cover $\{T^{-i}U:0\le i\le N_2-1\}$ of $X$.
	Let $N=\max \{N_1, 2N_2\}$.
	By the continuity of $T$, there exists $\delta_3>0$
	such that
	if $d(x,y)<\delta_3$ implies that
	$d(T^ix,T^iy)<\frac{\ep}{4}$ for any $0\le i\le N$.
	Put $\delta=\min\{\delta_2,\delta_3\}$.
	Let $x,y\in X$ with $d(x,y)<\delta$ and $n\in\N$.
	If $n\leq N$, then
	\[\frac{1}{n}\sum_{i=0}^{n-1}d(T^ix,T^iy)\leq \frac{1}{n} \cdot n\cdot \frac{\ep}{4} <\ep.\]
	If $n>N$, there exists $0\leq i_0\leq N_2-1$ such that
	$x,y\in T^{-i_0}U$, i.e., $T^{i_0}x,T^{i_0}y\in U$, and then
	\begin{align*}
	\frac{1}{n}\sum_{i=0}^{n-1}d(T^ix,T^iy)
	&\leq  \frac{1}{n}\sum_{i=0}^{i_0-1}d(T^ix,T^iy)
	+ \frac{1}{n}\sum_{i=0}^{n-1}d(T^{i}T^{i_0}x,T^{i}T^{i_0}y)\\
	&\leq \frac{\ep}{4}+ \frac{1}{n}\sum_{i=0}^{n-1}d(T^{i}T^{i_0}x,T^{i}z)+
	\frac{1}{n}\sum_{i=0}^{n-1}d(T^{i}T^{i_0}x,T^{i}z)\\
	&\leq \frac{\ep}{4}+\frac{\ep}{4}+\frac{\ep}{4}<\ep.
	\end{align*}
	Therefore $\hat{d}_n(x,y)<\ep$ for all $n\in \Z$.
	This implies that $(X,T)$ is equicontinuous in the mean.
\end{proof}

\begin{rem}
When this paper was finished, we became aware of the work of \cite{QZ18} that Qiu and Zhao can show
that in general a t.d.s.\ is mean equicontinuous if and only if
it is equicontinuous in the mean.
\end{rem}

\subsection{Topological complexity
	with respect to  $\{\bar{d}_n$\}}
For $n\in\mathbb{N}$ and $x,y\in X$, define
\[\bar{d}_n(x,y)= \frac{1}{n}\sum_{i=0}^{n-1}d(T^ix,T^iy).\]
It is easy to see that for each $n\in\N$, $\bar{d}_n$
is a metric on $X$ which is topologically equivalent to
the metric $d$.
For $x\in X$ and $\ep>0$, let
$B_{\bar{d}_n}(x,\ep)=\{y\in X\colon \bar{d}_n(x,y)<\ep\}.$
For $n\in\N$ and $\ep>0$, define
\[\overline{\spann}_K(n,\varepsilon)=
\min\biggl\{\#(F)\colon
F\subset K\subset \bigcup_{x\in F}B_{\bar{d}_n}(x,\ep)\biggr\}.\]
We say that a subset $K$ of $X$ has \emph{bounded topological complexity
with respect to $\{\bar{d}_n$\}}
if for every $\ep>0$ there exists  a positive integer $C=C(\ep)$
such that $\overline{\spann}_K(n,\varepsilon)\leq C$
for all $n\geq 1$.

As $ \bar{d}_n(x,y)\leq \hat{d}_n(x,y)$ for all $n\in\N$ and $x,y\in X$,
if $K$ has bounded topological complexity
with respect to $\{\hat{d}_n\}$ then it  also has bounded topological complexity
with respect to $\{\bar{d}_n\}$.
Intuitively, dynamical systems with
bounded topological complexity with respect to $\{\bar{d}_n\}$
have similar properties of ones with respect to $\{\hat{d}_n\}$ or $\{d_n\}$.
But we will see that this is far from being true.
The key point is that the sequence $\{\bar{d}_n\}$ of metrics may be not monotonous.
If a dynamical system has bounded topological complexity
with respect to $\{\bar{d}_n\}$,
then by Theorem~\ref{thm:mu-discrete-spectrum} in next section,  every invariant measure has discrete spectrum.
So it is simple in the measure-theoretic sense.
But we have the following proposition which is a surprise in some sense. Since the construction is somewhat long and complicated,
we move it to the Appendix.

\begin{prop}\label{prop:key-example1}
There is a distal, non-equicontinuous, non-uniquely ergodic,
minimal system,
which has bounded topological complexity
with respect to $\{\bar{d}_n\}$.
\end{prop}

We can modify the example in Proposition~\ref{prop:key-example1}
to be uniquely ergodic and also present the construction
in the Appendix.
\begin{prop}\label{prop:key-example2}
	There is a distal, non-equicontinuous, uniquely ergodic,
	minimal system,
	which has bounded topological complexity
	with respect to $\{\bar{d}_n\}$.
\end{prop}

\begin{rem}
As each distal mean equicontinuous minimal system is equicontinuous,
the systems constructed in  Propositions~\ref{prop:key-example1}
and \ref{prop:key-example2} are not mean equicontinuous.
\end{rem}

We have a natural question.

\begin{ques} \label{want-2}Is there a non-trivial weakly mixing, even strongly mixing minimal system with bounded topological complexity with respect to $\{\bar d_n\}$?
\end{ques}

We are just informed by Huang and Xu \cite{HXU} the above question has an affirmative answer for weakly mixing minimal systems.
The question if there is
a non-trivial strongly mixing  minimal  system with bounded topological complexity with respect to $\{\bar d_n\}$ is still open.

\section{Invariant measures with bounded measure-theoretic complexity}
In this section, we will study
the measure-theoretic complexity of invariant (Borel probability) measures with respect to three kinds of metrics.
\subsection{Measure-theoretic complexity with respect to $\{d_n\}$}
Let $(X,T)$ be a t.d.s.\ and $\mu$ be
an invariant measure on $(X,T)$.
For $n\in\N$ and $\ep>0$,
let
\[\spann_\mu(n,\ep)=
\min\biggl\{\#(F)\colon F\subset X \text{ and }
\mu\biggl(\bigcup_{x\in F} B_{d_n}(x,\ep)\biggr)>1-\ep\biggr\}.\]
Recall that this is the same notion defined in \cite{Katok} by Katok. We say that $\mu$ has \emph{bounded  complexity
with respect to $\{d_n\}$}
if for every $\ep>0$ there exists a positive integer $C=C(\ep)$ such that
$\spann_\mu(n,\varepsilon)\leq C$ for all $n\geq 1$.

We will show that an invariant measure with bounded complexity
with respect to $\{d_n\}$
is equivalent to the $\mu$-equicontinuity property.

\begin{thm}\label{thm-mu-equi}
Let $(X,T)$ be a t.d.s.\ and
$\mu$ be an invariant measure on $(X,T)$.
Then $\mu$ has bounded complexity
with respect to $\{d_n\}$  if and only if
 $T$ is $\mu$-equicontinuous.
\end{thm}

\begin{proof}
($\Leftarrow$)
First assume that $(X,T)$ is $\mu$-equicontinuous.
Fix $\varepsilon>0$.
There exists a $T$-equicontinuous measurable subset $K$ of $X$ with $\mu(K)>1-\varepsilon$.
As the measure $\mu$ is regular, we can require the set $K$ to be compact.
Now the result follows from Theorem~\ref{thm:T-equicontinuous-subset},
as $\spann_\mu(n,\varepsilon)\leq \spann_K(n,\varepsilon)$.

\medskip

($\Rightarrow$)
For any $\tau>0$, we need to find a $T$-equicontinuous set $K$ with $\mu(K)>1-\tau$.
Now fix $\tau>0$.
As $\mu$ has bounded complexity with respect to $\{d_n\}$,
for any $M>0$, there exists $C=C_M>0$ such that
for every $n\geq 1$
there exists a subset $F_n$  of $X$ with $\#(F_n)\leq C$
such that
\[\mu\biggl(\bigcup_{x\in F_n} B_{d_n}\Bigl(x,\tfrac{1}{M}\Bigr)\biggr)>1-\frac{\tau}{2^{M+2}}.\]
As the measure $\mu$ is regular,
pick a compact subset $K_n$ of $\bigcup_{x\in F_n} B_{d_n}(x,\frac{1}{M})$
with $\mu(K_n)>1-\frac{\tau}{2^{M+2}}$.
Without loss of generality, assume that $F_n\to F_M$, $K_n\to K_M$ as $n\to\infty$ in the Hausdorff metric.
Then $\#( F_M)\leq C$.
As $K_n$ is closed,
\[\mu(K_M)\geq \limsup_{n\to\infty}\mu(K_n)\geq 1-\frac{\tau}{2^{M+2}}.\]
For any $x\in K_M$ and $n\in \N$,
there exists an $N>0$ such that for any $k>N$
there exists $x_k \in K_k$ and $y_k\in F_k$
such that $d_n(x,x_k)<\frac{1}{M}$ and $d_k(x_k,y_k)<\frac{1}{M}$.
Without loss of generality, assume that $y_k\to y$ as $k\to\infty$.
Then $y\in F_M$.
By the monotonicity of $\{d_n\}$, we have
\[d_n(x,y_k)\leq d_n(x,x_k)+d_n(x_k,y_k)\leq d_n(x,x_k)+d_k(x_k,y_k)\leq \tfrac{2}{M}.\]
Letting $k$ go to infinity, one has $d_n(x,y)\leq  \frac{2}{M}$.
Then $K_M\subset \bigcup_{x\in F_M} B_{d_n}(x,\frac{3}{M})$
and $\spann_{K_M}(n,\frac{3}{M})\leq \#(F_M)\leq C_M$.

Let $K=\bigcap_{M=1}^\infty K_M$. Then $\mu(K)>1-\tau$ and for any $M\geq 1$,
\[\spann_{K}(n,\tfrac{3}{M})\leq \spann_{K_M}(n,\tfrac{3}{M})\le C_M\]
for all $n\geq 1$.
Now by Theorem~\ref{thm:T-equicontinuous-subset}, $K$ is $T$-equicontinuous.
This proves that $(X,T)$ is $\mu$-equicontinuous.
\end{proof}

\begin{rem}
Similar to the observation in Remark~\ref{rem:open-cover}, the open cover version of Theorem \ref{thm-mu-equi} was proved in \cite[Proposition 3.3]{HLY}.	
\end{rem}

\subsection{Measure-theoretic complexity with respect to  $\{\hat{d}_n\}$}
For $n\in\N$ and $\ep>0$,
let
\[\widehat{\spann}_\mu(n,\ep)=
\min\biggl\{\#(F)\colon F\subset X \text{ and }
\mu\biggl(\bigcup_{x\in F} B_{\hat{d}_n}(x,\ep)\biggr)>1-\ep\biggr\}.\]
We say that $\mu$ has \emph{bounded complexity
with respect to  $\{\hat{d}_n\}$}
if for every $\ep>0$ there exists a positive integer $C=C(\ep)$ such that
$\widehat{\spann}_\mu(n,\varepsilon)\leq C$ for all $n\geq 1$.

We will show that an invariant measure with bounded complexity
with respect to  $\{\hat{d}_n\}$ is equivalent to the following
two kinds of measure-theoretic equicontinuity.
We say that $T$ is \emph{$\mu$-equicontinuous in the mean}
if for any $\tau>0$
there exists a  measurable subset $K$ of
$X$ with $\mu(K)>1-\tau$ which is equicontinuous in the mean,
and \emph{$\mu$-mean equicontinuous}
if for any $\tau>0$
there exists a  measurable subset $K$ of
$X$ with $\mu(K)>1-\tau$ which is mean equicontinuous.

\begin{thm}\label{thm:mu-eq-in-mean}
Let $(X,T)$ be a t.d.s.\ and
$\mu$ be an invariant measure on $(X,T)$.
Then the following statements are equivalent:
\begin{enumerate}
	\item $\mu$ has bounded  complexity
	with respect to  $\hat{d}_n$;
	\item 	$T$ is $\mu$-equicontinuous in the mean;
	\item $T$ is $\mu$-mean equicontinuous.
\end{enumerate}
\end{thm}

\begin{proof}

(1) $\Rightarrow$ (2)
Following the proof of Theorem \ref{thm-mu-equi}, we know that for a given $\tau>0$, there is a compact subset $K$
such that $\mu(K)\geq  1-\tau$ and for any $M\geq 1$,
$\widehat\spann_{K}(n,\frac{6}{M})\leq  C_M$ for all $n\geq 1$.
By Theorem~\ref{thm:strongly-mean-equi}, $K$ is  equicontinuous
in the mean.
This proves that $(X,T)$ is $\mu$-equicontinuous in the mean.

\medskip

(2) $\Rightarrow$ (3) is obvious.

\medskip
(3) $\Rightarrow (1)$ Now assume that $(X,T)$ is $\mu$-mean equicontinuous.
Fix $\varepsilon>0$. Then there is a compact $K=K(\ep)\subset X$ such that $\mu(K)>1-2\ep$ and $K$ is mean equicontinuous.
There exists a $\delta>0$ such that
\[
\limsup_{n\ra \infty} \frac{1}{n}\sum_{i=0}^{n-1}d(T^ix,T^iy)<\ep/4
\]
for all $x,y\in K$ with $d(x,y)<\delta$.
As $K$ is compact,
there exists a finite subset $F$ of $K$ such that
$K\subset \bigcup_{x\in F}B(x,\delta)$.
Enumerate $F$ as $\{x_1,x_2,\dotsc,x_m\}$.
For $j=1,\dotsc,m$ and $N\in\N$, let
\[A_N(x_j)=\biggl\{y\in B(x_j,\delta)\cap K\colon
\frac{1}{n}\sum_{i=0}^{n-1}d(T^ix_j,T^iy)<\ep/2,\
n=N,N+1,\dotsc\biggr\}.
\]
It is easy to see that for each $j=1,\dotsc,m$,
$\{A_N(x_j)\}_{N=1}^\infty$ is an increasing sequence and $B(x_j,\delta)\cap K=\bigcup_{N=1}^\infty A_N(x_j)$.
Choose $N_1\in\N$ and a compact subset $K_1$ of $A_{N_1}(x_1)$
such that $\mu(K_j)>\mu(B(x_j,\delta)\cap K)-\frac{\ep}{2m}$.
Choose $N_2\in\N$ and a compact subset $K_2$ of $A_{N_2}(x_2)$
such that $K_1\cap K_2=\emptyset$ and
$\mu(K_1\cup K_2)>
\mu((B(x_1,\delta)\cup B(x_2,\delta))\cap K)-\frac{2\ep}{2m}$.
By induction,
we can choose compact subsets $K_j$ of $A_{N_j}(x_j)$ for $j=1,\dotsc,m$
with $\mu(\bigcup_{j=1}^m K_j)> \mu(K)-\frac{\ep}{2}>1-\ep$
and $K_i\cap K_j=\emptyset$ for $1\leq i<j\leq m$.

Let $K_0=\bigcup_{j=1}^m K_j$ and
$N_0=\max\{N_j\colon j=1,2,\dotsc,m\}$.
There exists $\delta_1>0$ such that
for every $x,y\in K$ with $d(x,y)<\delta_1$
there exists $j\in \{1,2,\dotsc,m\}$ with $x,y\in K_j$.
By the continuity of $T$,
there exists $\delta_2>0$ such that $d_{N}(x,y)<\ep$
for every $x,y\in X$ with $d(x,y)<\delta_2$.
Let $\delta_3=\min\{\delta_1,\delta_2\}$.
By the compactness of $K_0$, there exists a finite subset $H$
of $K_0$ such that $H\subset \bigcup_{x\in H}B(x,\delta_3)$.
Fix $n\geq 1$ and $y\in K_0$.
There exists $x\in H$ with $d(x,y)<\delta_3$.
If $n<N_0$, then $\hat{d}_n(x,y)\leq d_{N_0}(x,y)<\ep$.
If $n\geq N_0$,  there exists $j\in\{1,2,\dotsc,m\}$
with $x,y\in K_j\subset A_{N_j}(x_j)$.
By the construction of $A_{N_j}(x_j)$ and $n\geq N_j$,
\[
\frac{1}{n}\sum_{i=0}^{n-1}d(T^ix,T^iy)\leq
\frac{1}{n}\sum_{i=0}^{n-1}d(T^ix,T^ix_j)+
\frac{1}{n}\sum_{i=0}^{n-1}d(T^ix_j,T^iy)
<\ep/2+\ep/2=\ep.
\]
For any $n\geq 1$, we have  $\hat{d}_n(x,y)<\ep$.
Then
\[K_0\subset \bigcup_{x\in H} B_{\hat{d}_n}(x,\ep)\]
and
\[\mu\biggl(\bigcup_{x\in H} B_{\hat{d}_n}(x,\ep)\biggr)
\geq \mu(K_0)>1-\ep.\]
This implies that $\widehat{\spann}_\mu(n,\ep)\leq \#(H)$
for all $n\geq 1$.
Then $\mu$ has bounded  complexity
with respect to  $\{\hat{d}_n\}$.
\end{proof}

\subsection{Measure-theoretic complexity with respect to  $\{\bar{d}_n\}$}
For $n\in\N$, $\ep>0$,
let
\[\overline{\spann}_\mu(n,\ep)=
\min\biggl\{\#(F)\colon F\subset X \text{ and }
\mu\biggl(\bigcup_{x\in F} B_{\bar{d}_n}(x,\ep)\biggr)>1-\ep\biggr\}.\]
We say that $\mu$ has \emph{bounded  complexity
with respect to  $\{\bar{d}_n\}$}
if for every $\ep>0$
there exists a positive integer $C=C(\ep)$ such that
$\overline{\spann}_\mu(n,\varepsilon)\leq C$ for all $n\geq 1$.

Unlike the topological case,
we can prove that bounded measure-theoretic complexity
with respect  $\{\bar{d}_n\}$ and $\{\hat{d}_n\}$ are equivalent.

\begin{thm}\label{nov-5}
Let $(X,T)$ be a t.d.s.\ and
$\mu$ be an invariant measure on $(X,T)$.
Then $\mu$ has bounded  complexity
with respect to  $\{\bar{d}_n\}$
if and only if it has bounded  complexity
with respect to  $\{\hat{d}_n\}$.
\end{thm}
\begin{proof}
It is clear that if $\mu$ has has bounded  complexity
with respect to  $\{\hat{d}_n\}$ then by definition it also has
has bounded  complexity
with respect to  $\{\bar{d}_n\}$.
	
Now assume that $\mu$ has bounded  complexity
with respect to  $\{\bar{d}_n\}$.
Let $\ep>0$. There is $C=C(\ep)$ such that for any $n\in\N$,
there is $F_n\in X$ with $\#(F_n)\leq C$  such that
\[
\mu\biggl(\bigcup_{x\in F_n} B_{\bar{d}_n}(x,\ep/8)\biggr)>1-\ep/8.
\]
By the Birkhoff pointwise ergodic theorem for $\mu\times \mu$ a.e. $(x,y)\in X^2$
$$\bar d_N(x,y)=\frac{1}{N}\sum_{i=0}^{N-1}d(T^ix,T^iy)\rightarrow d^*(x,y).$$
So for a given $0<r< \min\{1,\frac{\ep}{2C}\}$, by the Egorov's theorem there are $R\subset X^2$ with $\mu\times \mu(R)>1-r^2$ and $N_0\in\N$
such that if $(x,y)\in R$ then
$$|\bar d_n(x,y)-\bar d_{N_0}(x,y)|<r,\ \text{for}\ n\ge N_0.$$
By the Fubini's theorem there is $A\subset X$ such that $\mu(A)>1-r$ and for any $x\in A$,
$\mu(R_x)>1-r$, where
\[R_x=\{y\in X: (x,y)\in R\}.\]
Enumerate $F_{N_0}=\{x_1,x_2,\dotsc,x_m\}$. Then $m\leq C$.
Let $I=\{1\leq i\leq m\colon A\cap B_{\bar{d}_{N_0}}(x_i, \ep/8)
\neq\emptyset \}$.
Denote $\#(I)=m'$. Then $1\leq m'\leq m$.
For each $i\in I$, pick $y_i\in A\cap  B_{\bar{d}_{N_0}}(x_i, \ep/8)$.
Then we have $B_{\bar{d}_{N_0}}(x_i, \ep/8)\subset
B_{\bar{d}_{N_0}}(y_i, \ep/4)$ for all $i\in I$.
As
\begin{align*}
\mu\biggl( A \cap \bigcap_{i\in I} R_{y_i}
\cap \bigcup_{x\in F_{N_0}} B_{\bar{d}_{N_0}}(x, \ep/8)
\biggr)\geq 1-r-m'r-\ep/8>1-\ep,
\end{align*}
choose a compact subset
\[K \subset A \cap \bigcap_{i\in I} R_{y_i}
\cap \bigcup_{x\in F_{N_0}} B_{\bar{d}_{N_0}}(x, \ep/8)\] with $\mu(K)>1-\ep$.
If $x\in K$, there exists $i\in I$
such that $x\in R_{y_i}\cap B_{\bar{d}_{N_0}}(y_i, \ep/4)$.
Then $(y_i,x)\in R$.
By the construction of $R$,
for any $n\ge N_0$,
\[\bar d_n(x,y_i)=\bar d_n(y_i,x)\le
\bar d_{N_0}(y_i,x)+r<\ep/4+r<\ep/2.\]
Let $\delta_1>0$ be a Lebesgue number of the open cover of $K$
by $\{K\cap B_{\bar{d}_{N_0}}(y_i,\ep/4)\colon i\in I\}$.
By the continuity of $T$, there exists $0<\delta<\delta_1$ such that
if $d(x_1,x_2)<\delta$ then $d_{N_0}(x_1,x_2)<\ep$.
Let $x_1,x_2\in K$ with $d(x_1,x_2)<\delta$.
There is $i\in I$
such that $x_1,x_2\in A_{y_i}\cap B_{\bar{d}_{N_0}}(y_i,\ep/4)$.
Fix $n\geq 1$.
If $n<N_0$, $\bar d_n(x_1,x_2)\le d_{N_0}(x_1,x_2)<\ep$.
If $n\geq N_0$,
\[\bar d_n(x_1,x_2)\le \bar d_n(x_1,y_i)+\bar d_n(x_2,y_i) <\ep/2+2r<\ep.\]
Then $\hat{d}_n(x_1,x_2)<\ep$ for all $n\geq 1$.
By the compactness of $K$, there exists a finite subset $H$
of $K$ such that $K\subset \bigcup_{x\in H}B(x,\delta)$.
For any $n\geq 1$,
we have
\[K\subset \bigcup_{x\in H} B_{\hat{d}_n}(x,\ep)\]
and then
\[\mu\biggl(\bigcup_{x\in H} B_{\hat{d}_n}(x,\ep)\biggr)
\geq \mu(K)>1-\ep.\]
This implies that $\widehat{\spann}_\mu(n,\ep)\leq \#(H)$
for all $n\geq 1$.
Then $\mu$ has bounded  complexity
with respect to  $\{\hat{d}_n\}$.
\end{proof}

We can restate Proposition 4.1 of \cite{HWY} as follows.
\begin{prop}\label{prop:discrete-specturm-HWY}
Let $(X,T)$ be an invertible t.d.s.\ and
$\mu$ be an invariant measure on $(X,T)$.
If $\mu$ has discrete spectrum,
then it has
 bounded  complexity
with respect to  $\{\bar{d}_n\}$.
\end{prop}

It is conjectured in  \cite{HWY} that the converse of
Proposition~\ref{prop:discrete-specturm-HWY} is also true.
If $\mu$ is ergodic,  by \cite[Corollary 39]{Felip},
we know that $\mu$  has discrete spectrum and
if and only if $\mu$ is mean equicontinuous.
So by Theorem \ref{thm:mu-eq-in-mean},
if an ergodic measure $\mu$ has bounded complexity
with respect to  $\{\bar{d}_n\}$,
then it has discrete spectrum.
We will show in Theorem \ref{thm:mu-discrete-spectrum} that that in general the converse of
Proposition~\ref{prop:discrete-specturm-HWY} is also true.

The following result was proved in~\cite[Theorem 2.7]{L16},
see also \cite[Corollary 39]{Felip}.
Here we provide a different direct proof.

\begin{prop} \label{thm:ergodic-measure-not-null}
Let $(X,T)$ be a t.d.s.\ and
$\mu$ be an ergodic invariant measure on $(X,T)$.
If $\mu$ does not have  discrete spectrum,
then there exists $\alpha>0$ such that
for $\mu\times\mu$-almost every pair $(x,y)\in X\times X$,
\[\liminf_{n\to\infty}\frac{1}{n}\sum_{i=0}^{n-1} d(T^ix,T^iy)>\alpha.\]
\end{prop}

\begin{proof}
Let $\mathcal{B}_\mu$ be the completion of the Borel $\sigma$-algebra $\mathcal{B}_X$ of $X$ with respect to $\mu$.
Corresponding to the discrete part of the spectrum of the action of $T$, there exists a compact metric abelian group $(G,+)$ with
Haar measure $\nu$, an element $\tau$ of $G$ such that $(G,\mathcal{B}_\nu, \nu,S)$ is the Kronecker factor
of $(X,\mathcal{B}_\mu,\mu,T)$ with an associated factor map $\pi:X\rightarrow G$, where
$\mathcal{B}_\nu$ be the completion of the Borel $\sigma$-algebra  of $G$ with respect
to $\nu$ and $S$ is the translation by $\tau$ on $G$.

Let $\mu=\int_G \mu_z {\rm d} \nu(z)$ be the disintegration of the measure $\mu$ over $\nu$.
For $s\in G$, let
$$\lambda_s= \int_G \mu_z\times \mu_{z+s} \dd \nu(z).$$
It is a classical result (see e.g.\ \cite[$\S 4.3.1$ Theorem 18]{HK18}) that there is $G_0\subset G$ with $\nu(G_0)=1$ such that for every $s\in G_0$, the system $(X\times X, \lambda_s,T\times T)$ is ergodic and
$$\mu\times \mu=\int_G \lambda_s \dd \nu(s)$$
is the ergodic decomposition $\mu\times \mu$ under $T\times T$.

By the Birkhoff ergodic theorem,
the limit
\[\lim_{n\rightarrow +\infty} \bar d_n (x,y)\]
exists and equals to
$$\int_{X\times X} d(x_1,x_2) \dd \lambda_s(x_1,x_2)$$ for some $s=s(x,y)\in G_0$
for $\mu\times \mu$-a.e. $(x,y)\in X^2$.

Now it is sufficient to show that  if  $(X,\mathcal{B}_\mu,\mu,T)$  does not have  discrete spectrum, then there exists $\alpha>0$ such that
$\int_{X\times X} d(x_1,x_2) \dd \lambda_s(x_1,x_2)\ge \alpha$
for all $s\in G_0$.

As $X$ is compact, pick a countable dense subset $\{y_n\colon n\in\mathbb{N}\}$ in $X$.
For $z\in G$,
$$c(z):=\inf_{n\in\mathbb{N}} \int_X d(x,y_n) \dd \mu_z(x).$$
It is clear that $c(z)>0$ if and only if $\mu_z$ is not a Dirac measure. Moreover, $c(\cdot)$ is a non-negative measurable function on $G$.
Put $$\alpha:=\int_G c(z) \dd \nu(z).$$
Since $(X,\mathcal{B}_\mu,\mu,T)$ is ergodic and does not have discrete spectrum, by Rohlin's theorem
$\mu_z$ is not a Dirac measure for $\nu$-a.e. $z\in G$.
This means that $c(z)>0$ for $\nu$-a.e. $z\in G$ and thus $\alpha>0$.
For each $y\in X$, there exists a subsequence $\{n_i\}$
such that $y_{n_i}\to y$ as $i\to\infty$.
Then for each $x\in X$, $d(x,y_{n_i})\to d(x,y)$ as $i\to\infty$.
By the Lebesgue dominated convergence theorem,
for each $z\in G_0$,
\[
\int_X d(x,y)\dd\mu_z(x) =
\lim_{i\to\infty} \int_X d(x,y_{n_i}) \dd\mu_z(x)\geq c(z).
\]
Thus, for each $s\in G_0$,
\begin{align*}
\int_{X\times X} d(x_1,x_2) \dd \lambda_s(x_1,x_2)
&= \int_G \Bigl(\int_{X\times X} d(x,y) \dd \mu_z\times \mu_{z+s}(x,y) \Bigr) {\rm d} \nu(z)\\
&= \int_G \Bigl(\int_X \Bigl( \int_X d(x,y) \dd\mu_z(x)\Bigr) \dd \mu_{z+s}(y) \Bigr) \dd \nu(z)\\
&\ge  \int_G \int_X c(z) {\rm d} \mu_{z+s}(y)  \dd \nu(z)\\
&= \int_G c(z)\dd\nu(z)=\alpha>0.
\end{align*}
This finishes the proof.
\end{proof}

Now we are able to show the converse of Proposition~\ref{prop:discrete-specturm-HWY}.

\begin{thm}\label{thm:mu-discrete-spectrum}
Let $(X,T)$ be an invertible t.d.s.\ and
$\mu$ be an invariant measure on $(X,T)$.
If $\mu$ has bounded complexity
with respect to  $\{\bar{d}_n\}$, then it has discrete spectrum.
\end{thm}

\begin{proof}
Let $A$ be the collection of points $z\in X$
which are generic to some ergodic measure,
that is,
for each $z\in A$, $\frac{1}{n}\sum_{i=0}^{n-1}\delta_{T^iz}\to \mu_z$
as $n\to\infty$
and $\mu_z$ is ergodic.
Then $A$ is measurable and $\mu(A)=1$.
We first prove the following Claim.

\medskip
\noindent\textbf{Claim 1}: $\mu_z$ has discrete spectrum for $\mu$-a.e.\ $z\in A$.
\begin{proof}[Proof of the Claim 1]
Let $A_1=\{z\in A\colon \mu_z$ does not has discrete spectrum$\}$.
We need to prove that $A_1$ is measurable and has zero $\mu$-measure.
The ergodic decomposition of $\mu$ can be expressed as
 $\mu=\int_A \mu_z \dd\mu(z)$ (see e.g.\ \cite[Theorem 6.4]{M87}).
For $k\in\N$ and $z\in A$, put
\[ F_k(z)=
\mu_z\times\mu_z\biggl(\biggl\{(x,y)\in X\times X\colon
\liminf_{n\to\infty}\frac{1}{n}\sum_{i=0}^{n-1} d(T^ix,T^iy)>\frac{1}{k}
\biggr\}\biggr).\]
As $\int_A \mu_z \times \mu_z d\mu(z)$ is an invariant measure on $(X\times X,T\times T)$,
for each $k\in\N$, $F_k$ is a measurable function on $A$.
By Theorem~\ref{thm:ergodic-measure-not-null},
we know that
$\mu_z$ does not have discrete spectrum if and only if there exists
$k\in\N$ such that $F_k(z)=1$.
Then $A_1= \bigcup_{k\in\N}\{z\in A\colon F_k(z)=1\}$
and it is measurable.
Now it is sufficient to prove $\mu(G_1)=0$.
If not, then $\mu(A_1)>0$ and there exists $k\in\N$
such that $\mu(\{z\in A\colon F_k(z)=1\})>0$.
Let $A_2=\{z\in A\colon F_k(z)=1\}$ and put $r=\mu(A_2)$. Then for every $z\in A_2$ and
for $\mu_z\times\mu_z$-a.e.\ $(x,y)\in X\times X$,
\begin{equation}
\liminf_{n\to\infty}\frac{1}{n}\sum_{i=0}^{n-1} d(T^ix,T^iy)>\frac{1}{k}.
\label{jian-sen-33}
\end{equation}

By	Theorems~\ref{thm:mu-eq-in-mean} and \ref{nov-5},
$(X,T)$ is $\mu$-mean equicontinuous.
Then there exists $M\subset X$ with $\mu(M)>1-\frac{r^2}{4}$
such that $M$ is mean equicontinuous.
By regularity of $\mu$, we can assume that $M$ is compact and $M\subset A$.
Let $A_3=\{z\in A\colon \mu_z(M)>1-\frac{r}{2}\}$.
Then $A_3$ is measurable, as $\mu=\int_A \mu_z \dd\mu(z)$ is the ergodic decomposition of $\mu$.
We have
\begin{align*}
1-\frac{r^2}{4}<\mu(M)=\int_A \mu_z(M)\dd\mu(z)
&\leq \int_{A_3} \mu_z(M)\dd\mu(z) +\int_{A\setminus A_3} \mu_z(M)\dd\mu(z) \\
&\leq \mu(A_3)+ (1-\mu(A_3))(1-\frac{r}{2}),
\end{align*}
which implies that $\mu(A_3)>1-\frac{r}{2}$.
Then $\mu(A_2\cap A_3)>r+(1-\frac{r}{2})-1=\frac{r}{2}>0$.
Pick $z\in A_2\cap A_3$. As  $M$ is mean equicontinuous,
there exists a $\delta>0$ such that
for any $x,y\in M$ with $d(x,y)<\delta$,
\[
\limsup_{n\to\infty}\frac{1}{n}\sum_{i=0}^{n-1}d(T^ix,T^iy)<\frac{1}{k}.
\]
As $M$ is compact, there exists a finite open cover
$\{U_1,U_2,\dotsc, U_m\}$ of $M$ with diameter less than $\delta$.
Since $z\in A_3$, $\mu_z(M)>1-\frac{r}{2}$.
Then there exists $i\in \{1,\dotsc,m\}$ such that $\mu_z(U_i)>0$
and also $\mu_z\times\mu_z(U_i)>0$.
Note that the diameter of $U_i$ is less than $\delta$, so
for any $x,y\in U_i$,
\[
\limsup_{n\to\infty}\frac{1}{n}\sum_{i=0}^{n-1}d(T^ix,T^iy)<\frac{1}{k},
\]
which contradicts to (\ref{jian-sen-33}). This ends the proof of Claim 1.
\end{proof}

Let
$$A_0=\{z\in A\colon  \mu_z\ \text{ has discrete spectrum}\}.$$
By Claim 1, we have $\mu(A_0)=1$.
Let $f\in C(X)$ be a  Lipschitz continuous function on $X$.
Then there exists $C >0$ such that $|f(x)-f(y)|\leq C d(x,y)$ for all $x,y\in X$.

Recall that the associated operator $U\colon L^2(\mu)\to L^2(\mu)$ is defined by $Uf=f\circ T$ for all $f\in L^2(\mu)$.
Inspired by the idea of  \cite[Theorem 1]{S82}, we have the following Claim.

\medskip
\noindent\textbf{Claim 2}:
For any $\tau>0$,
there exists $M^*\in\mathcal{B}$ with $\mu(M^*)>1-\tau$
such that $f\cdot \mathbf{1}_{M^*}$ is almost periodic,
i.e., $\{U^{n}(f\cdot \mathbf{1}_{M^*})\colon n\in\mathbb{Z}\}$ is precompact in $L^2(\mu)$.
\begin{proof}[Proof of the Claim 2]
By	Theorems~\ref{thm:mu-eq-in-mean} and \ref{nov-5},
$(X,T)$ is $\mu$-mean equicontinuous.
Fix a constant $\tau>0$.
Then there exists $M\subset X$ with $\mu(M)>1-\tau$
such that $M$ is mean equicontinuous.
Let $M^*=\bigcup_{n\in\mathbb{Z}} T^{-n}M$.
To show that $f\cdot \mathbf{1}_{M^*}$ is almost periodic,
we only need to prove for any sequence $\{t_n\}$ in $\mathbb{Z}$
there exists a subsequence $\{s_n\}$ of $\{t_n\}$  such that
$\{U^{s_n}(f\cdot \mathbf{1}_{M^*})\}$  is a Cauchy sequence in $L^2(\mu)$.

\medskip
By regularity of $\mu$, we can assume that $M$ is compact and $M\subset A_0$.
Choose a countable dense subset $\{z_m\}$ in $M$.
As $\mu_{z_1}$ has discrete spectrum, there exists a subsequence $\{t_{n,1}\}$ of $\{t_n\}$ such that
$\{U^{t_{n,1}}f\colon n\in\N \}$ is a Cauchy sequence in $L^2(\mu_{z_1})$. Inductively assume that
for each $i\le m-1$ we have defined $\{t_{n,i}\}$ (which is a subsequence of $\{t_{n,{i-1}}\}$) such that
$\{U^{t_{n,i}}f\colon n\in\N \}$ is a Cauchy sequence in $L^2(\mu_{z_i})$.
As $\mu_{z_m}$ has discrete spectrum,
there exists a subsequence $\{t_{n,m}\}$ of $\{t_{n,m-1}\}$ such that
$\{U^{t_{n,m}}f\colon n\in \N \}$ is a
Cauchy sequence in $L^2(\mu_{z_m})$.
Let $s_n=t_{n,n}$ for $n\geq 1$.
By the usual diagonal procedure,
$\{U^{s_n}f\colon n\in \N \}$ is a
Cauchy sequence in $L^2(\mu_{z_m})$ for all $m\geq 1$.

\medskip

Fix $\ep>0$. As $M$ is mean equicontinuous in $(X,T)$,
there exists $\delta>0$ such that for any $x,y\in M$ with $d(x,y)<\delta$,
\[\limsup_{n\to\infty}\frac{1}{n}\sum_{i=0}^{n-1}\Bigl( d(T^ix,T^iy) \Bigr)^2<\ep.\]
Fix $z\in M$. There exists $m\in \N$ such that $d(z,z_m)<\delta$.
For any $j\neq k\in \N$,
\begin{align*}
\Vert U^{s_j}f-U^{s_k}f\Vert_{L^2(\mu_z)}^2&
=\int_X |U^{s_j}f-U^{s_k}f|^2 \dd\mu_z
=\lim_{n\to\infty}\frac{1}{n}\sum_{i=0}^{n-1}|f(T^{s_j+i}z)-f(T^{s_k+i}z)|^2\\
&\leq C^2 \lim_{n\to\infty}\frac{1}{n}\sum_{i=0}^{n-1}\Bigl(d(T^{s_j+i}z,T^{s_k+i}z)\Bigr)^2\\
&\leq C^2\biggl( \limsup_{n\to\infty}\frac{1}{n}\sum_{i=0}^{n-1}\Bigl(d(T^{s_j+i}z,T^{s_j+i}z_m)\Bigr)^2\\
&\qquad\qquad \qquad+\limsup_{n\to\infty}\frac{1}{n}\sum_{i=0}^{n-1}\Bigl(d(T^{s_k+i}z,T^{s_k+i}z_m)\Bigr)^2\\
&\qquad\qquad\qquad +\lim_{n\to\infty}\frac{1}{n}\sum_{i=0}^{n-1}\Bigl(d(T^{s_j+i}z_m,T^{s_j+i}z_m)\Bigr)^2\biggr)\\
&\leq C^2\Bigl(2\ep + \Vert U^{s_j}f-U^{s_k}f\Vert_{L^2(\mu_{z_m})}^2\Bigr).
\end{align*}
As $\{U^{s_n}f\colon n\in \N \}$ is a
Cauchy sequence in $L^2(\mu_{z_m})$ for all $m\geq 1$,
$\{U^{s_n}f\colon n\in \N \}$ is also a
Cauchy sequence in $L^2(\mu_{z})$.
Then for each $z\in M$,
\[ \lim_{N\to\infty}\sup_{j,k\geq N} \int_X |U^{s_j}f-U^{s_k}f|^2 \dd\mu_z =0. \]
For each $y\in M^*$, there exists $n\in\mathbb{Z}$ and $z\in M$
such that $T^nz=y$.
Then $\mu_z=\mu_y$.
For $z\in  M^*$, put
\[f_N(z)=\sup_{j,k\geq N} \int_X |U^{s_j}f-U^{s_k}f|^2 \dd\mu_z.\]
By the dominated convergence theorem,
\begin{equation}\label{dmt}
\lim_{N\to\infty}\int_{M^*}f_N(z) \dd\mu(z)=\int_{M^*}\lim_{N\to\infty} f_N(x)\dd\mu(z)=0.
\end{equation}
It is easy to see that
\begin{align*}
\sup_{j,k\geq N}  \int_{M^*} \int_X |U^{s_j}f-U^{s_k}f|^2 \dd\mu_z \dd\mu(z)&
\le  \int_{M^*} \biggl(\sup_{j,k\geq N} \int_X |U^{s_j}f-U^{s_k}f|^2 \dd\mu_z\biggr) \dd\mu(z)\\
&=\int_{M^*}f_N(z) \dd\mu(z),
\end{align*}
from where we deduce
 \[ \lim_{N\to\infty}
\biggl(\sup_{j,k\geq N}  \int_{M^*} \int_X |U^{s_j}f-U^{s_k}f|^2 \dd\mu_z \dd\mu(z)\biggr)=0\ \  \ \text{by}\ (\ref{dmt}).\]
As $\int_{M^*}g \dd \mu=\int_{M^*}(\int g \dd\mu_z) \dd\mu(z)$ for any $g\in L^2(\mu)$ we have
 \[ \lim_{N\to\infty}
\biggl(\sup_{j,k\geq N}  \int_{M^*}  |U^{s_j}f-U^{s_k}f|^2 \dd\mu(z)\biggr)=0. \]
Note that $T(M^*)=M^*$, so
\[\int_{M^*} |U^{s_j}f-U^{s_k}f|^2 \dd\mu(z)=
\int |U^{s_j}(f\cdot \mathbf{1}_{M^*}) -U^{s_k}(f\cdot \mathbf{1}_{M^*})|^2 \dd\mu.\]
Thus $\{U^{s_n}(f\cdot \mathbf{1}_{M^*})\}$  is a Cauchy sequence in $L^2(\mu)$,
which ends the proof of Claim 2.
\end{proof}

Note that
The collection of almost periodic functions $g$ is closed in $L^2(\mu)$.
As the measure of ${M^*}$ in Claim 2 can be arbitrary close to $1$,
$f$ is also an almost periodic function in  $L^2(\mu)$.
As the collection of Lipschitz continuous functions in dense in $C(X)$ (see e.g. \cite[Theorem 11.2.4.]{D})
and $C(X)$ is dense in $L^2(\mu)$,
then for every function $g\in L^2(\mu)$ is almost periodic in $L^2(\mu)$,
that is $\mu$  has discrete spectrum.
\end{proof}

\begin{rem}
After we have  finished this paper,
Nhan-Phu Chung informed us that
Theorem~\ref{thm:mu-discrete-spectrum} was also proved in \cite[Theorem~3.2]{VZP12} with a different method.
Note that an invariant measure $\mu$ has bounded complexity
with respect to  $\{\bar{d}_n\}$ in our sense
if and only if every $\ep>0$, the scaling sequences
with respect to $\mu$ and $d$ are bounded as in \cite[Definition~3.1]{VZP12}.
It should be noticed that in the introduction of  \cite{VZP12}
it requires a mild condition that the standard (Lebesgue) space $(X,\mu)$ is non-atomic.
\end{rem}

In Theorem~\ref{thm:mu-discrete-spectrum},
we show that if an invariant measure $\mu$ of a t.d.s.\ $(X,T)$ has bounded complexity with respect to $\{\bar{d}_n\}$, then almost all the ergodic
components in the ergodic decomposition of $\mu$ have discrete spectrum.
In the following remark we provide an example which shows that
it may happen there are uncountably many pairwise non-isomorphic ergodic components in the ergodic decomposition, and the set of unions of all eigenvalues of the ergodic components are countable.

\begin{rem}
The space $X$ is the product  $ \{0,1\}^\N \times (S^1) ^{\N}$.
Let $\{\tau_i\colon  i\in \N\}$
be a family of irrational numbers independent over
the rational numbers.
The measure $\mu$ is the product of the Bernoulli measure $(\frac{1}{2}, \frac{1}{2})$ on $\{0,1\}^\N $ and
the product measure $\lambda_\N$ on $(S^1) ^{\N}$,
where each coordinate is equipped with the Lebesgue measure $\lambda$.

The transformation $T:X\ra X$ is defined in the following way: let $\omega= ( \omega_i)_{i\geq 1}  \in \{0,1\}^\N$ and $w=(w)_{i\geq 1}\in (S^1)^{\N}$.
Define $T(\omega,w)=(\omega, w')$, where $(w')_i=w_i$ if $\omega_i=0$ and $(w')_i=T_iw_i$ if $\omega_i=1$, where $T_i$  is the translation by $\tau_i$ on $(S^1)_i$.
It is easy to see that $\{\omega\} \times (S^1) ^{\N}$ is $T$-invariant for any $\omega\in \{0,1\}^\N$.

 Let the distance on $X$ be the sum of the distances
 $$d_1(\omega , \omega')= \sum_{i\geq 1} \frac{1}{2^i}\mid \omega_i -\omega'_i \mid\
 \text{and}\ d_2 (s,s')= \sum_{i\geq 1} \frac{1}{2^i} d'(s_i, s'_i),$$ where $d'$ is the distance on the circle $S^1$, so that $d((\omega,s),
 (\omega',s'))= d_1(\omega , \omega') + d_2 (s,s')$.

 It is not difficult to see that $T$  has bounded  complexity
with respect to  $\{\bar{d}_n\}$.
 Note that the ergodic components are $\{\omega\} \times (w', \Pi_{i \mid\omega{_i=1}} (S^1)_i)$, where $w'\in \Pi_{i \mid\omega{_i=0}} (S^1)_i.$
\end{rem}

\begin{rem}
Assume that $(X,T)$ is a minimal system with bounded complexity with respect to $\{\bar d_n\}$ for an invariant measure $\mu$.
It is interesting to know whether almost all the ergodic measures in the ergodic decomposition of $\mu$ are isomorphic.
After the first draft version of this paper was finished, Cyr and Kra informed us in~\cite{CK18} that there exists an example dose not satisfy the condition, see Proposition~\ref{prop:Cyr-Kra} in Appendix B.
\end{rem}

\appendix
\section{Two examples}
The aim of this  appendix is to construct two examples
announced in Section 3. We remark that the measure complexity for a minimal distal system can be very complicated,
see for example \cite{HXY}.

\subsection{The construction of the system in Proposition~\ref{prop:key-example1}}
We view the unit circle $\mathbb{T}$ as $\mathbb{R}/\mathbb{Z}$
and also as $[0,1)\pmod 1$.
For $a\in \R$ we let $\Vert a\Vert =\min\{|a-z|: z\in\mathbb{Z}\}$ which induces a distance on $\mathbb{T}$.
Let $\alpha\in\mathbb{R}\setminus\mathbb{Q}$ be an irrational number and $R_\alpha: \mathbb{T}\rightarrow \mathbb{T},x\rightarrow x+\alpha$
the rotation on $\mathbb{T}$ by $\alpha$.  In this subsection we will construct a skew product map $T:\T^2\rightarrow \T^2$
with $T(x,y)=(x+\alpha,y+h(x))$ for any $x,y\in\T$, where $h:\T\rightarrow \mathbb{R}$ is continuous and will be defined below.

\medskip
The general idea to construct $h$ is the following. First we choose some disjoint intervals (see (\ref{disjoint11intervals})) in $\T$
and define $h_1:\T\rightarrow \mathbb{R}$ such that $h_1$ takes positive values in some intervals,
negative values in the rest intervals and zero in rest points (see (\ref{eq-first1}) below). This results that
$\sum_{i=0}^{n-1}h_1(R_\alpha^ix)$ is small under some conditions (see Lemmas \ref{lm} and \ref{lm-2}).
Then we choose smaller disjoint intervals (see (\ref{disjoint-2})) in $\T$ and define $h_2:\T\rightarrow \mathbb{R}$ such that $h_2$ takes smaller
positive values in some intervals, smaller negative values in the rest intervals and zero in the rest pionts. We do this for each $h_k$
and finally we put $h=\sum_{k=1}^\infty h_k$. Note that  for $n\in\mathbb{N}$ and $x,y\in \mathbb{T}$,
$T^n(x,y)= \Bigl(R_\alpha^nx,y+\sum_{i=0}^{n-1}h(R_\alpha^ix)\Bigr).$
When calculating $\sum_{i=0}^{n-1}h(R_\alpha^ix)$, we only need
to take care $\sum_{i=0}^{n-1}(\sum_{j=1}^kh_j)(R_\alpha^ix)$ when $k$ large enough (see (\ref{k-infty})). By careful choosing $h_k$
we can show the minimality, non-equicontinuity and non-unique ergodicity. Now let us begin the construction.

\medskip

Let $\eta=\frac{1}{100}$, $M_1=10$ and $N_1=10M_1$.
As $\alpha$ is irrational,
the two-side orbit $\{n\alpha\colon n\in\mathbb{Z}\}$
of $0$ under the rotation $R_\alpha$ are pairwise distinct.
Choose $\delta_1>0$ small enough such that
the intervals
$$ [i\alpha-\delta_1,i\alpha +\delta_1],\ i=-1,0,1,\cdots,2N_1$$
are pairwise disjoint on $\mathbb{T}$.
Put
\begin{equation}\label{disjoint11intervals}
E_1=\bigcup_{i=0}^{2N_1-1} [i\alpha-\delta_1,i\alpha+\delta_1],
\end{equation}
and  $$F_1=\{i\alpha-\delta_1,i\alpha+\delta_1 \colon
i=0,1,\cdots,2N_1-1\}.$$

The total length of intervals in $E_1$ is $4N_1\delta_1$.
Shrinking $\delta_1$ if necessary, we can require $4N_1\delta_1<\eta/2$.
Put $$l_1=\min\{\Vert x-y\Vert :x,y\in F_1, x\neq y\}$$
and $\gamma_0=2l_1$.

For $k=2,3,4,\cdots$, we will define
$M_k, N_k,\delta_k, E_k, F_k, l_k$ and $\gamma_{k-1}$  by induction.
Assume that $M_{k-1},N_{k-1},\delta_{k-1}, E_{k-1}$, $F_{k-1}, l_{k-1}$ and $\gamma_{k-2}$ have been defined
such that the total length of intervals in  $E_{k-1}$ is less than $\frac{\eta}{2^{k-1}}$.
As $R_\alpha$ is uniquely ergodic on $\mathbb{T}$,
choose $M_k>N_{k-1}$ large enough such that
for any $x,y\in\mathbb{T}$, one has
\begin{equation}\label{20-1}
\{0\le i\le M_k-1\colon R_\alpha^ix\in(y,y+l_{k-1})\}\neq\emptyset.
\end{equation}
and for any $n\ge M_k$ and any $x\in\mathbb{T}$,
\begin{equation}\label{20-2}
\frac{1}{n}\#(\{0\le i\le n-1\colon
R_\alpha^ix \in E_{k-1}\})<\frac{\eta}{2^{k-1}}.
\end{equation}
Let $N_k=10^kM_k$.
Choose $\delta_k>0$ small enough such that
$$\{i\alpha \pm\delta_{k} \colon i=0,1,2,\cdots,2N_{k}-1\}\cap F_{k-1}=\emptyset,$$
and
$$[i\alpha-\delta_k,i\alpha+ \delta_k],\ i=-1,0,1,\cdots,2N_k$$
are pairwise disjoint intervals on $\mathbb{T}$.
Choose $0<\gamma_{k-1}<\delta_{k-1}$ small enough such that
\begin{align}
\label{fc-1}
[i\alpha-\gamma_{k-1},i\alpha+\gamma_{k-1}],\
-2N_{k}\le i\le 2N_{k}+2N_{k-1}
\end{align}
are pairwise disjoint intervals on $\mathbb{T}$.
Put
\begin{equation}\label{disjoint-2}
E_k=\bigcup_{i=0}^{2N_k-1}[i\alpha-\delta_k,i\alpha+\delta_k]
\end{equation}
and
\begin{equation}
F_k=F_{k-1}\cup\{i\alpha+\delta_k,i\alpha-\delta_k
\colon i=0,1,\cdots,2N_k-1\}
\end{equation}
The total length of intervals in  $E_{k}$ is $4N_k\delta_k$.
Shrinking $\delta_k$ if necessary,
we can require $4N_k\delta_k<\frac{\eta}{2^{k}}$.
Let
\begin{align}\label{lk-def}
l_k=\min\biggl(\Bigl\{\Vert x-y\Vert \colon x,y\in F_k, x\neq y\Bigr\}\cup\Bigl\{\frac{\gamma_i}{2k^2}:i=1,2,\cdots,k-1\Bigr\}\biggr).
\end{align}
This finishes the induction.

\medskip
For each $k\in\N$, define $h_{k}^*,h_k:\R\rightarrow [-1/2,1/2)$ such that
\[
h_{k}^*(x) =
\begin{cases}
\frac{1}{N_{k}}(1-|\frac{x-m}{\gamma_{k}}|), & \text{for } x\in[m-\gamma_{k},m+\gamma_{k}]\textrm{ with }m\in\mathbb{Z},\\
0,& \text{otherwise},
\end{cases}
\]
and
\[h_{k}(x)=\sum_{i=0}^{N_{k}-1}h_{k}^*(x-i\alpha)-\sum_{i=N_{k}}^{2N_{k}-1}h_{k}^*(x-i\alpha).\]
As the intervals in $E_k$ are pairwise disjoint and $\gamma_k<\delta_k$,
it is easy to check that for any $x\in \mathbb{R}$,
\begin{align}\label{eq-first1}
h_{k}(x) = \begin{cases}
h_{k}^*(x-i\alpha), & \text{if } x\in [i\alpha-\gamma_{k},i\alpha+\gamma_{k}] \pmod 1,\  i=0,1,2,\cdots,N_{k}-1,\\
-h_{k}^*(x-i\alpha), & \text{if } x\in [i\alpha-\gamma_{k},i\alpha+\gamma_{k}] \pmod 1,\ i=N_{k},\cdots,2N_{k}-1,\\
0,& \text{otherwise}.
\end{cases}
\end{align}
In particular, $h_k(x)=0$ for $x\not\in E_k \pmod 1$,
$$h_k(i\alpha)=
\begin{cases}
\frac{1}{N_{k}}, &\text{ for }  i=0,1,\dotsc, N_{k}-1,\\
-\frac{1}{N_{k}}, &\text{ for } i=N_{k},N_k+1,\dotsc,2N_{k}-1,
\end{cases}$$
and for any $z\in [-\delta_x,\delta_x]$,
\begin{equation}
\sum_{s=0}^{2N_k-1}h_k(R^s_\alpha z)=0.
\label{eq:hk-2nk-1=0}
\end{equation}
It is also easy to see that for any $x\in\mathbb{R}$,
\begin{equation}
|h_k(x)|\leq \frac{1}{N_k}=\frac{1}{10^kM_k}<\frac{1}{10^k},
\label{eq:hk-bounded}
\end{equation}
and  $h_k$ is Lipschitz continuous with a Lipschitz constant $\frac{1}{N_k\gamma_k}$, that is, for any $x,y\in\mathbb{R}$,
\begin{equation}
|h_k(x)-h_k(y)|\leq \frac{1}{N_k\gamma_k}|x-y|. \label{eq:Lip-continuous}
\end{equation}
For any $x\in \R$, we have $h_k(x+1)=h_k(x)$,
so we can regard $h_{k}$ as a function from $\mathbb{T}$ to $\mathbb{R}$.
Now,
define $h\colon \mathbb{T}\rightarrow \mathbb{R}$
as for each $x\in \mathbb{T}$
$$h(x)=\sum_{k=1}^\infty h_k(x).$$
It is easy to see that $h$ is continuous since
\[\sum_{k=1}^\infty |h_k(x)|\leq\sum_{k=1}^\infty \frac{1}{10^k}<\frac{1}{2}.\]
For $k\geq 1$,
we set
\begin{equation}
h_{1,k}(x)=\sum_{i=1}^{k}h_i(x)\text{ and } h_{k,\infty}(x)=\sum_{i=k}^{\infty}h_i(x).
\end{equation}
Then
$$h(x)=h_{1,k}(x)+h_{k+1,\infty}(x)$$
and
\begin{equation}\label{k-infty}
\Vert h_{k,\infty}(x)\Vert
\leq \sum_{i=k}^\infty \Vert h_{i}(x)\Vert
\leq\sum_{i=k}^\infty \frac{1}{10^iM_i}
\leq \frac{1}{M_k} \sum_{i=k}^\infty \frac{1}{10^i}
=\frac{1}{M_k} \frac{1}{9\cdot 10^{k-1}}<\frac{1}{9\cdot 10^{k-1}}.
\end{equation}
Finally, we define a skew product map as follows:
\[T\colon \mathbb{T}^2\to \mathbb{T}^2,\quad
(x,y)\mapsto (x+\alpha,y+h(x)).\]
It is clear that $T$ is continuous.
We will show that the system $(\mathbb{T}^2,T)$ is as required.
By the definition, it is clear that  $(\mathbb{T}^2,T)$
is distal.

For any real function $g$ on $\T$ and $x\in \mathbb{T}$, we set $H_0^g\equiv 0$
and
$$H_n^{g}(x):=\sum_{i=0}^{n-1}g(R_\alpha^ix)$$
for $n\geq 1$.
Recall that $h(x)=\sum_{k=1}^\infty h_k(x)$, so
$$H_n^{h}(x)=\sum_{i=1}^{\infty }H_n^{h_i}(x)
=H_n^{h_{1,k}}(x)+H_n^{h_{k+1,\infty}}(x).
$$
We choose a compatible metric $d$ on $\T^2$ by
$$d((x_1,y_1),(x_2,y_2)):=\Vert x_1-x_2\Vert +\Vert  y_1-y_2 \Vert,$$
for any $(x_1,y_1),(x_2,y_2)\in \mathbb{T}^2$.
We remark that for $n\in\mathbb{N}$ and $x,y\in \mathbb{T}$
$$
T^n(x,y)= \Bigl(R_\alpha^nx,y+\sum_{i=0}^{n-1}h(R_\alpha^ix)\Bigr)=
(R_\alpha^nx,y+H_n^h(x)).$$
To estimate the orbit of $(x,y)$, the key point is to control $H_n^h(x)$. The following two lemmas will be useful in the estimation.
\begin{lem}\label{lm}
	Assume $x\in\mathbb{T}$, $i\in \mathbb{N}$ and $m\in\mathbb{N}$.
	If $x, R_\alpha^mx\in E_i^c\cup[-\delta_i,\delta_i]$,
	then one has $$H_m^{h_i}(x)=0.$$
\end{lem}
\begin{proof}
	Let $J=\{0\le j\le m-1: R_\alpha^jx\in
	[-\delta_i,\delta_i]\}$.
	We first claim that
	$$ \{0\le k \le m-1:R_\alpha^k x\in E_i\}=
	\bigcup_{j\in J}\{j+l:0\le l\le 2N_i-1\}.$$
	To see this equality firstly we note that if $j\in J$,
	then $R_\alpha^jx\in [-\delta_i,\delta_i]$.
	This implies that
	$R_\alpha^{j+l}x\in[l\alpha-\delta_i,l\alpha+\delta_i]\subset E_i$
	for $0\le l\le 2N_i-1$.
	Since $j\le m-1$ and $R_\alpha^mx\in E_i^c\cup[\delta_i,\delta_i]$,
	one has $j+ 2N_i-1\leq m$ and then
	$\{j+l:0\le l\le 2N_i-1\}\subset\{0,1,2,\cdots,m-1\}$.
	Thus,
	$\{0\le k \le m-1:R_\alpha^k x\in E_i\}\supset \cup_{j\in J}\{j+l:0\le l\le 2N_i-1\}.$
	
	Conversely, if $k\in \{0,1,2,\cdots,m-1\}$ with $R^k_\alpha x\in E_i$. This means that $R^k_\alpha x\in[s\alpha-\delta_i,s\alpha+\delta_i]$
	for some $0\le s\le   2N_i-1$.
	If $k<s$, then
	$x\in [(s-k)\alpha-\delta_i,(s-k)\alpha+\delta_i]$, which  contradicts to the assumption $x\in E_i^c\cup[-\delta_i,\delta_i]$.
	This implies $k\ge s$.
	Hence we have $k-s\in J$ and $k\in \{(k-s)+l:0\le l\le 2N_i-1\}$.
	Thus we get $\{0\le k\le m-1:R_\alpha^kx\in E_i\}\subset \cup_{j\in J}\{j+l:0\le l\le 2N_i-1\}$. This
	proves the claim.
	
	By the claim, we have
	\begin{align*}
	H_{m}^{h_i}(x)&=\sum_{\substack{0\le k\le m-1\\R_\alpha^kx\notin E_i}}h_i(R_\alpha^kx)+\sum_{\substack{0\le k\le m-1\\R_\alpha^kx\in E_i}}h_i(R_\alpha^kx)\\
	&=0+\sum_{j\in J}\sum_{l=0}^{2N_i-1}h_i(R_\alpha^{l}(R_\alpha^jx))=0. \qquad\text{by \eqref{eq:hk-2nk-1=0}}
	\end{align*}
	This finishes the proof of Lemma \ref{lm}.
\end{proof}

\begin{lem}\label{lm-2}
	Assume $x\in\mathbb{T}$,  $m, k\in\N$ and $1\le j\le k-1$.
	If $\Vert m\alpha\Vert <l_k$ and $x,R_\alpha^mx\in [i\alpha-\delta_j,i\alpha+\delta_j]$ for some $0\le i\le 2N_j-1$,  then one has $$\Vert H_m^{h_j}(x)\Vert <\frac{1}{k^2}.$$
\end{lem}
\begin{proof}
	First, by \eqref{lk-def}, one has
	\begin{align*}
	l_k&\le \min\{\Vert (i\alpha+\delta_k)-(j\alpha+\delta_k)\Vert :0\le i<j\le 2N_k-1\}\\
	&=\min_{0\le r\le2N_k-1}\Vert r\alpha\Vert .
	\end{align*}
	Thus, $m\ge2N_k$ since $\Vert m\alpha\Vert <l_k$. Next, by the construction of $E_j$, one has
	$$R_\alpha^{2N_j-i}x,R_\alpha^{m-i}x=R_\alpha^{m-2N_j}(R_\alpha^{2N_j-i}x)\in E_j^c\cup[-\delta_j,\delta_j]$$
	since $x,R_\alpha^mx\in [i\alpha-\delta_j,i\alpha+\delta_j]$.
	By Lemma \ref{lm}, one has $H^{h_j}_{m-2N_j}(R_\alpha^{2N_j-i}x)=0$ and
	\begin{align*}
	H_{m}^{h_j}(x)&=H^{h_j}_{2N_j-i}(x)+H^{h_j}_{m-2N_j}(R_\alpha^{2N_j-i}x)+H^{h_j}_{i}(R_\alpha^{m-i}x)\\
	&=H^{h_j}_{2N_j-i}(x)+H^{h_j}_{i}(R_\alpha^{m-i}x)\\
	&=(H^{h_j}_{2N_j-i}(x)+H^{h_j}_{i}(R_\alpha^{-i}x))+(H^{h_j}_{i}(R_\alpha^{m-i}x)-H^{h_j}_{i}(R_\alpha^{-i}x))\\
	&=H^{h_j}_{2N_j}(R_\alpha^{-i}x)+(H^{h_j}_{i}(R_\alpha^{m-i}x)-H^{h_j}_{i}(R_\alpha^{-i}x)).
	\end{align*}
	Notice that $R_\alpha^{-i}x\in [-\delta_j,\delta_j]$. By \eqref{eq:hk-2nk-1=0}, one has $$H^{h_j}_{2N_j}(R_\alpha^{-i}x)=\sum_{s=0}^{2N_j-1}h_j(R_\alpha^{s}(R_\alpha^{-i}x))=0.$$
	This implies that
	\begin{align*}
	\Vert H_{m}^{h_j}(x)\Vert &\le\Vert H^{h_j}_{i}(R_\alpha^{m-i}x)-H^{h_j}_{i}(R_\alpha^{-i}x)\Vert \\
	&\le\sum_{s=0}^{i-1}\Vert h_j(R_\alpha^{m-i+s}x)-h_j(R_\alpha^{-i+s}x)\Vert \\
	&\le i \cdot l_k\cdot\frac{1}{N_j\gamma_j},
	\end{align*}
	where the last inequality follows from  \eqref{eq:Lip-continuous} and $\Vert R_\alpha^{-i+s}x-R_\alpha^{m-i+s}x\Vert= \Vert m\alpha\Vert<l_k$ for $s=0,1,2\cdots,i-1$.
	Finally, by \eqref{lk-def},
	\begin{align*}
	\Vert H_{m}^{h_j}(x)\Vert \le 2N_j\cdot \frac{\gamma_j}{2k^2}\cdot \frac{1}{N_j\gamma_j}=\frac{1}{k^2}.
	\end{align*}
	This finishes the proof of Lemma \ref{lm-2}.
\end{proof}

\begin{prop}
	$(\mathbb{T}^2,T)$ is minimal.
\end{prop}

\begin{proof}
We need to show  every point $(x,y)$ has a dense orbit.
Fix $(x,y)\in\mathbb{T}^2$, $0<\epsilon<1$ and $k\in\mathbb{N}$.
There exists $n_1\in\mathbb{N}$ such that $R_\alpha^{n_1}x\in[-\epsilon\gamma_k,\epsilon\gamma_k]$.
Let $(x_1,y_1)=T^{n_1}(x,y)$. Then $x_1=R^{n_1}_\alpha x$
and $\Vert x_1\Vert \leq \epsilon\gamma_k$.
	
Now fix $(x',y')\in\mathbb{T}^2$.
Note that points in $F_{k-1}$ divide the unite circle into open arcs
with length not less than $l_{k-1}$.
The collection of these arcs is denoted by $\mathcal{F}_{k-1}$.  There exists $(a_1,a_2)\in\mathcal{F}_{k-1}$ such that $x'\in [a_1,a_2)$.
As $(a_1,a_2)\cap F_{k-1}=\emptyset$,
either $(a_1,a_2)\subset E_{k-1}$ or $(a_1,a_2)\subset E_{k-1}^c$.
If $(a_1,a_2)\subset E_{k-1}$, then $[a_1,a_2)\subset [j\alpha-\delta_{k-1},j\alpha+\delta_{k-1}]$ for some $0\leq j\leq 2N_{k-1}-1$,
and we take $a=j\alpha+\delta_{k-1}$.
If $(a_1,a_2)\subset E_{k-1}^c$, we take $a\in [a_1,a_2)$ such that $x'\in[a,a+l_{k-1})\subset[a_1,a_2)$ since the length of $[a_1,a_2)$ is not less than $l_{k-1}$.	
Note that in any case $(a,a+l_{k-1})$ is a subset of some
$(b_1,b_2)\in \mathcal{F}_{k-1}$.
For any $1\le i\le k-2$, as $F_{i}\subset F_{k-1}$,
$(b_1,b_2)\cap E_{i}=\emptyset$.
Then $(b_1,b_2)$ is either a subset of $[j\alpha-\delta_i,j\alpha+\delta_i]$ for some $0\le j\le 2N_i-1$ or a subset of $E_i^c$. Summing up the above arguments, one has:
\begin{itemize}
\item[(i)] $(a,a+l_{k-1})\subset E_{k-1}^c$
	and $\min\{\Vert x'-x''\Vert \colon x''\in(a,a+l_{k-1})\}\le 2\delta_{k-1}$;
\item[(ii)]
	for all $1\le i\le k-2$, either $(a,a+l_{k-1})\subset [j\alpha-\delta_i,j\alpha+\delta_i]$ for some $0\le j\le 2N_i-1$ or $(a,a+l_{k-1})\subset E_i^c$.
\end{itemize}
	
By \eqref{20-1}, there exists an integer $n_2\in[0,M_k)$
such that
$R_\alpha^{n_2}(x_1)\in (a,a+l_{k-1})$.
Suppose
$$y'-(y_1+H_{n_2}^{h}(x_1))=b\, (\text{mod } 1).$$
Then $b\in[0,1).$
By \eqref{20-1}, there exists an integer $n_3\in[[10^kb]M_k,([10^kb]+1)M_k)$ such that $n_3\ge n_2$ and $R_\alpha^{n_3}(x_1)\in (a,a+l_{k-1})$.
Note that $n_3< 10^kM_k=N_k$ and  $$b-\frac{2}{10^k}\le\frac{([10^kb]-1)M_k}{N_k}\le\frac{n_3-n_2}{N_k}\le\frac{([10^kb]+1)M_k}{N_k}\le b+\frac{2}{10^k}.$$
By (i) and Lemma \ref{lm}, one has $H_{n_3-n_2}^{h_{k-1}}(R_\alpha^{n_2}x_1)=0$.
By (ii) and Lemmas \ref{lm} and \ref{lm-2}, one has
	$$\Vert H_{n_3-n_2}^{h_i}(R_\alpha^{n_2}x_1)\Vert <\frac{1}{(k-1)^2}$$
for $1\le i\le k-2$.
Thus, one has
\begin{align*}
	\Vert y'-(y_1+H_{n_3}^h(x_1))\Vert
	&=\Vert y'-(y_1+H_{n_2}^{h}(x_1))-H_{n_3-n_2}^{h}(R_\alpha^{n_2}x_1)\Vert \\
	&= \Bigl \Vert b-\sum_{i=k}^\infty H_{n_3-n_2}^{h_i}(R_\alpha^{n_2}x_1)-\sum_{i=1}^{k-2} H_{n_3-n_2}^{h_i}(R_\alpha^{n_2}x_1) \Bigr\Vert \\
	&\le\Vert b-H_{n_3-n_2}^{h_k}(R_\alpha^{n_2}x_1)\Vert \\
	&\qquad\qquad + \sum_{i=k+1}^\infty \Vert H_{n_3-n_2}^{h_i}(R_\alpha^{n_2}x_1)\Vert +\sum_{i=1}^{k-2} \Vert H_{n_3-n_2}^{h_i}(R_\alpha^{n_2}x_1)\Vert \\
	&\le\Vert b-H_{n_3-n_2}^{h_k}(n_2\alpha)\Vert +\Vert H_{n_3-n_2}^{h_k}(n_2\alpha)-H_{n_3-n_2}^{h_k}(n_2\alpha+x_1)\Vert \\
	&\qquad\qquad +\sum_{i=k+1}^\infty\frac{n_3-n_2}{N_i}+(k-2)\cdot \frac{1}{(k-1)^2}\\
	&\le\Bigl\Vert b-(n_3-n_2)\frac{1}{N_k}\Bigr\Vert +\epsilon\frac{n_3-n_2}{N_k}
	+\sum_{i=k+1}^\infty\frac{1}{10^i}+\frac{1}{k-1}\\
	&\le \frac{3}{10^k}+2\epsilon+\frac{1}{k-1}.
\end{align*}
We deduces
\begin{align*}
	d((x',y'),T^{n_3+n_1}(x,y))&=d((x',y'),T^{n_3}(x_1,y_1))\\
	&=\Vert x'-R_\alpha^{n_3}x_1\Vert +\Vert y'-(y_1+H_{n_3}^h(x_1))\Vert \\
	&\le l_{k-1}+2\delta_{k-1}+\frac{3}{10^k}+2\epsilon+\frac{1}{k-1}.
\end{align*}
This implies that $(x',y')\in \overline{\mathrm{Orb}((x,y),T)}$ if we let $k\rightarrow+\infty$ and $\epsilon\rightarrow 0$.
Hence $(\mathbb{T}^2,T)$ is minimal.
\end{proof}

For $1\le j\le k$, we let
\[E_{j,k}=\bigcup_{i=j}^kE_i. \]
By \eqref{20-2}, for any $n\geq M_{k+1}$ and $x\in \mathbb{T}$,
\begin{equation} \label{eq:Ejk}
\frac{1}{n}\#(\{0\le i\le n-1\colon R_\alpha^ix \in E_{j,k}\})
<\sum_{i=j}^k \frac{\eta}{2^{i}}<\frac{\eta}{2^{j-1}}<\frac{1}{2}.
\end{equation}

\begin{prop}
	$(\mathbb{T}^2,T)$ is not equicontinuous.
\end{prop}
\begin{proof}
	To show that $(\mathbb{T}^2,T)$ is not equicontinuous,
	it is sufficient to show that for any $\epsilon>0$,
	there exist $(x_1,y_1), (x_2,y_2)\in\mathbb{T}^2$
	and $n\in\mathbb{N}$ such that $d((x_1,y_1),(x_2,y_2))\le\epsilon$ and $d(T^n(x_1,y_1),T^n(x_2,y_2))\ge\frac{1}{9}$.
	
	Fix  $\epsilon>0$.
	There exists $k\in\mathbb{N}$ such that $l_k+\delta_k<\epsilon$.
	Put $x'=\delta_k+\frac{1}{2}l_k$.
	One has $R_\alpha^{i}x'\in E_k^c$ and $h_k(R_\alpha^{i}x')=0$ for $i=0,1,\cdots,N_k-1$.
	By \eqref{eq:Ejk}, we can choose integers $n_1\in [0,M_k-1]$ and
	$n_2\in [\frac{1}{2}N_k-M_k,\frac{1}{2}N_k-1]$
	such that
	\[R_\alpha^{n_1}0,R_\alpha^{n_1}x', R_\alpha^{n_2}0,R_\alpha^{n_2}x'\in E_{1,k-1}^c. \]
	By using Lemma \ref{lm} and the fact $R_\alpha^{n_1}x',R_\alpha^{n_2}x'\in E_k^c$, we have
	\begin{align*}
	H_{n_2-n_1}^{h}(R_\alpha^{n_1}0)
	&=H_{n_2-n_1}^{h_{1,k-1}}(R_\alpha^{n_1}0)+H_{n_2-n_1}^{h_k}(R_\alpha^{n_1}0)+H_{n_2-n_1}^{h_{k+1,\infty}}(R_\alpha^{n_1}0)\\
	&= H_{n_2-n_1}^{h_k}(R_\alpha^{n_1}0)+H_{n_2-n_1}^{h_{k+1,\infty}}(R_\alpha^{n_1}0)\\
	&=(n_2-n_1)\frac{1}{N_k}+H_{n_2-n_1}^{h_{k+1,\infty}}(R_\alpha^{n_1}0)
	\end{align*}
	and
	\begin{align*}
	H_{n_2-n_1}^{h}(R_\alpha^{n_1}x')
	&=H_{n_2-n_1}^{h_{1,k-1}}(R_\alpha^{n_1}x')+
	H_{n_2-n_1}^{h_k}(R_\alpha^{n_1}x')
	+H_{n_2-n_1}^{h_{k+1,\infty}}(R_\alpha^{n_1}x')\\
	&=H_{n_2-n_1}^{h_{k+1,\infty}}(R_\alpha^{n_1}x').
	\end{align*}
	Note that $\frac{1}{2}N_k-2M_k\le n_2-n_1\le \frac{1}{2}N_k$ and
	$N_k=10^kM_k$, so
	\[ \frac{2}{5}\leq (n_2-n_1)\frac{1}{N_k}\leq \frac{1}{2}.\]
	By \eqref{eq:hk-bounded}, we have
	\begin{align*}
	\Vert H_{n_2-n_1}^{h_{k+1,\infty}}(R_\alpha^{n_1}0)-H_{n_2-n_1}^{h_{k+1,\infty}}(R_\alpha^{n_1}x')\Vert&\le \sum_{i=k+1}^\infty\Bigl(\Vert H_{n_2-n_1}^{h_i}(R_\alpha^{n_1}0)\Vert +\Vert H_{n_2-n_1}^{h_i}(R_\alpha^{n_1}x')\Vert \Bigr)\\
	&\le \sum_{i=k+1}^\infty 2(n_2-n_1)\frac{1}{N_i}\le \sum_{i=k+1}^\infty \frac{2N_k}{N_i}\\
	&\le2\sum_{i=k+1}^\infty\frac{1}{10^{i-k}}
	=\frac{2}{9}.
	\end{align*}
	Thus,
	$$\frac{16}{90}\le \Vert H_{n_2-n_1}^{h}(R_\alpha^{n_1}0)-H_{n_2-n_1}^{h}(R_\alpha^{n_1}x')\Vert
	\le \frac{65}{90},$$
	and
	\begin{align*}
	&d(T^{n_2-n_1}(R_\alpha^{n_1}0,0),T^{n_2-n_1}(R_\alpha^{n_1}x',0))\\	&\qquad\qquad =d((R_\alpha^{n_2}0,H^h_{n_2-n_1}(R_\alpha^{n_1}0)),
	(R_\alpha^{n_2}x',H^h_{n_2-n_1}(R_\alpha^{n_1}x'))\\
	&\qquad\qquad \ge
	\Vert H_{n_2-n_1}^{h}(R_\alpha^{n_1}0)-H_{n_2-n_1}^{h}(R_\alpha^{n_1}x') \Vert \ge \frac{16}{90}\ge \frac{1}{9}
	\end{align*}
	with $$d((R_\alpha^{n_1}0,0),(R_\alpha^{n_1}x',0))=
	\Vert R_\alpha^{n_1}0-R_\alpha^{n_1}x' \Vert = \Vert x'\Vert  =\delta_k+\frac{1}{2}\ell_k\le\epsilon.$$
	This implies that $(\mathbb{T}^2,T)$ is not equicontinuous.
\end{proof}

\begin{prop}\label{prop-11-14-1}
	$(\mathbb{T}^2,T)$ is not uniquely ergodic.
\end{prop}
\begin{proof}
	Let $m_{\mathbb{T}^2}$ be the unique normalized Haar measure on $\mathbb{T}^2$.
	For any $m_{\mathbb{T}^2}$-integrable function $f(x,y)$, by the Fubini's theorem, one has
	\begin{align*}
	\int_{\mathbb{T}^2}f\circ T(x,y)\dd m_{\mathbb{T}^2}
	&=\int_{\mathbb{T}}\int_{\mathbb{T}}f(R_\alpha x,y+h(x))\dd m_\mathbb{T}(y)\dd m_\mathbb{T}(x)\\
	&=\int_{\mathbb{T}}\int_{\mathbb{T}}f(R_\alpha x,y)\dd m_\mathbb{T}(y)\dd m_\mathbb{T}(x)\\
	&=\int_{\mathbb{T}}\int_{\mathbb{T}}f(x,y)\dd m_\mathbb{T}(y)\dd m_\mathbb{T}(x)\\
	&=\int_{\mathbb{T}^2}f( x,y)\dd m_{\mathbb{T}^2}.
	\end{align*}
	Therefore $m_{\mathbb{T}^2}$ is $T$-invariant.
	
	If $(\mathbb{T}^2,T)$ is uniquely ergodic,
	then $m_{\mathbb{T}^2}$ is the unique invariant measure.
	We take a measurable function
	\begin{align*}
	f(x,y) =\mathbf{1}_{\mathbb{T}\times[0,\frac{1}{2})}(x,y)
	-\mathbf{1}_{\mathbb{T}\times[\frac{1}{2},1)}(x,y).
	\end{align*}
	Note that the boundary of $\mathbb{T}\times[0,\frac{1}{2}) $
	and $\mathbb{T}\times[\frac{1}{2},1)$ have zero $m_{\mathbb{T}^2}$-measure.
	By unique ergodicity, we have
	\begin{align}
	\label{11-9}
	\lim_{n\to\infty}\frac{1}{n}\sum_{i=0}^{n-1}f(T^i(x,y))=\int_{\mathbb{T}^2}f\dd m_{\mathbb{T}^2}=0
	\end{align}
	for each $(x,y)\in \mathbb{T}^2$.
	
	For $k\ge 1$, put $$A_k=
	\Bigl\{s\in\Bigl\{\frac{1}{10}N_k,\frac{1}{10}N_k+1,\cdots,\frac{4}{10}N_k\Bigr\} \colon R_\alpha^s0=s\alpha\in\bigcap_{j=1}^{k-1}E_j^c\Bigr\}.$$
	For $i\in A_k$, it is clear that
	$0,R^{i}_\alpha0\in \bigcap_{j=1}^{k-1}E_j^c\cup[-\delta_j,\delta_j]$,
	so we have that by  Lemma \ref{lm}
	\begin{align*}
	H^{h}_i(0)-\frac{i}{N_k}=\sum_{j=1}^\infty H^{h_j}_i(0)-\frac{i}{N_k}=\sum_{j=k}^\infty H^{h_j}_i(0)-\frac{i}{N_k}=\sum_{j=k+1}^\infty H^{h_j}_i(0)
	\end{align*}
	where in the last equality we use the fact that $H^{h_k}_i(0)=\sum_{l=0}^{i-1}h_k(R_\alpha^l0)=
	\frac{i}{N_k}$ by \eqref{eq-first1}.
	Notice that
	\[\Vert H^{h_j}_i(0)\Vert \le\sum_{l=0}^{i-1}\Vert h_j(R_\alpha^l0)\Vert \le\frac{i}{N_j}.\]
	Therefore
	\begin{align*}
	\Bigl\Vert H^{h}_i(0)-\frac{i}{N_k}\Bigr\Vert &=
	\Bigl\Vert \sum_{j=k+1}^\infty H^{h_j}_i(0)\Bigr\Vert \le \sum_{j=k+1}^\infty\Vert H^{h_j}_i(0)\Vert \le  \sum_{j=k+1}^\infty\frac{i}{N_j}\\
	&\le \sum_{j=k+1}^\infty\frac{i}{10^jM_j}\leq  \sum_{j=k+1}^\infty\frac{1}{10^j}<\frac{1}{10}.
	\end{align*}
	It is clear that $\frac{1}{10}\leq \frac{i}{N_k}\leq \frac{4}{10}$.
	So  $H^{h}_i(0)\in[0,\frac{1}{2})$ and
	\begin{align}
	\label{11-9-3}
	f(T^i(0,0))=f((i\alpha,H^{h}_i(0)))=1.
	\end{align}
	Put $S_k=\{0,1,\cdots,\frac{1}{2}N_k-1\}$ and
	$B_k=\{s\in S_k: R_\alpha^s0\in E_{1,k-1}\}$,
	and by the construction \eqref{20-2}
	\begin{align*}
	\frac{\#(B_k)}{\frac{1}{2}N_k}\le
	\sum_{j=1}^{k-1}\frac{\eta}{2^{j}}<\eta=\frac{1}{100}.
	\end{align*}
	Hence
	$$\frac{\#(A_k)}{\frac{1}{2}N_k}\ge
	\frac{\frac{3}{10}N_k-\#B_k}{\frac{1}{2}N_k}>\frac{59}{100}.$$
	Since $f(T^i(0,0))=1$ for $i\in A_k$ and $f(T^i(0,0))\in\{-1,1\}$ for $i\in S_k\setminus A_k$,
	we have
	\begin{align*} \frac{1}{\frac{1}{2}N_k}\sum_{i=0}^{\frac{1}{2}N_k-1}f(T^i(0,0))
	&=\frac{1}{\frac{1}{2}N_k}\Bigl(\sum_{i\in A_k}f(T^i(0,0))+\sum_{i\in S_k\setminus A_k}f(T^i(0,0))\Bigr)\\
	&\ge \frac{1}{\frac{1}{2}N_k}(\#(A_k)-\#(S_k\setminus A_k))\\
	&\ge\frac{59}{100}-\frac{41}{100}=\frac{18}{100}.	
	\end{align*}
	Thus
	$$\limsup_{k\to\infty}\frac{1}{\frac{1}{2}N_k}\sum_{i=0}^{\frac{1}{2}N_k-1}f(T^i(0,0))\ge \frac{18}{100}>0,$$
	which contradicts \eqref{11-9}.
	Therefore $(\mathbb{T}^2,T)$ is not uniquely ergodic.
	This completes the proof.
\end{proof}

For any real function $g$ on $\T$, $n\in\mathbb{N}$ and $x,y\in X$, we set
$$\bar{d}_n^g(x,y):=\frac{1}{n}\sum_{m=0}^{n-1}
\Vert H_m^{g}(x)-H_m^{g}(y)\Vert.$$
Then for any $(x_1,y_1), (x_2,y_2)\in \T^2$, we have
\begin{align*}
\bar{d}_n((x_1,y_1), (x_2,y_2))
&=\frac{1}{n}\sum_{m=0}^{n-1}d((R_\alpha^m x_1, y_1+H_m^h(x_1)),
(R_\alpha^m x_2, y_1+H_m^h(x_2)))\\
&\le \Vert x_1-x_2\Vert
+\Vert y_1-y_2\Vert +\bar d_{n}^h(x_1,x_2).
\end{align*}

The main result of this subsection is as follows.
\begin{prop}\label{pro-1}
	$(\mathbb{T}^2,T)$ has bounded topological complexity
	with respect to $\{\bar{d}_n\}$.
\end{prop}

\begin{proof}
	It is sufficient to show that for any $\epsilon\in(0, \frac{1}{100})$,  there exist two constants $C(\epsilon)>0$ and
$K(\epsilon)\in\mathbb{N}$ such that $\overline{\spann}(n,17\epsilon)\le C(\epsilon)$ for any $n> K(\epsilon)$.
	
	\medskip
	First, we choose an integer $q\in\mathbb{N}$ such that
	\begin{eqnarray}\label{eq-34}
	\sum_{i=q+1}^\infty \frac{\eta}{2^i}<\epsilon \text{ and } \frac{1}{10^q}<\epsilon.
	\end{eqnarray}
	Then there exists $\delta(\epsilon)>0$ such that
	\begin{eqnarray}\label{eq-40}
	\sum_{i=1}^{q}\Vert H_{s}^{h_{i}}(x)-H_{s}^{h_{i}}(y)\Vert <\epsilon,
	\end{eqnarray}
	for any $0\le s\le M_{q+1}-1$ and any $x,y\in \mathbb{T}$ with $\Vert x-y\Vert <\delta(\epsilon)$.
	
	Put $c_\epsilon=\lceil\frac{1}{\epsilon}\rceil$ and $c_\delta=\lceil \frac{1}{\delta(\epsilon)}\rceil$.
	Let $$C(\epsilon)=100 c_\epsilon^{11}c_\delta
	\text{ and }K(\epsilon)=2N_{q+2}.$$
	In the following, we are going to show that
	for any $n> K(\epsilon)$  there exists  a cover $\mathcal{T}$ of $\mathbb{T}^2$ (that depends on $n$),
	such that
	$$\#(\mathcal{T})\le C(\epsilon) \text{ and } \bar{d}_{n}((x_1,y_1),(x_2,y_2))\le 17\epsilon$$
	for any $(x_1,y_1),(x_2,y_2)\in W\in\mathcal{T}$. This will imply $\overline{\spann}(n,17\epsilon)\le C(\epsilon)$ for any $n> K(\epsilon)$.
	
	\medskip
	Now fix an integer $n> K(\epsilon)$.
	There exists a unique integer $k\ge q+2$
	such that
	\[2N_{k}<n\le2N_{k+1}.\]
Recall that $$h=h_{1,k-1}+h_{k}+h_{k+1}+h_{k+2,\infty},$$
$$\bar{d}_n((x_1,y_1), (x_2,y_2))\le \Vert x_1-x_2\Vert +\Vert y_1-y_2\Vert +\bar{d}_{n}^h(x_1,x_2),$$
and
$$\bar{d}_{n}^h(x_1,x_2)\le \bar{d}_{n}^{h_{1,k-1}}(x_1,x_2)+\bar{d}_{n}^{h_k}(x_1,x_2)+\bar{d}_{n}^{h_{k+1}}(x_1,x_2)+\bar{d}_{n}^{h_{k+2,\infty}}(x_1,x_2).$$
	
We divide the remaining proof into four steps, bounding each term of the sum above.
	
\medskip
\noindent \textbf{Step 1}: \textit{We will construct a finite cover $\mathcal{P}$ of $\mathbb{T}$ such that
$$\#(\mathcal{P})\le c_\delta c_\epsilon^2 \text{ and } \bar{d}_{n}^{h_{1,k-1}}(x,y)<6\epsilon$$
for  $x,y\in P\in \mathcal{P}$. }
	
\medskip
Firstly, for any $x\in\mathbb{T}$ and $\ell\ge 2$, we define
\[
n_\ell^{*}(x)=\min\{i\geq 0: R_\alpha^ix\in E_{1,\ell-1}^c\}	\ \text{and}\  x_\ell^{*}=R_\alpha^{n_\ell^{*}(x)}x.\]
Clearly, $n_\ell^{*}(x)\le M_\ell-1$ by \eqref{20-1}.
By Lemma \ref{lm}, if $R_\alpha^mx\in E_i^c$ for some
$1\le i\le \ell-1$ and $m\ge M_\ell$, one has $H_{m-n_\ell^*(x)}^{h_i}(x_\ell^*)=0$ and then
	\begin{eqnarray}\label{eq-bh1}
	H_{m}^{h_i}(x)=H_{n_\ell^*(x)}^{h_i}(x)+H_{m-n_\ell^*(x)}^{h_i}(x_\ell^*)=H_{n_\ell^*(x)}^{h_i}(x).
	\end{eqnarray}
	
	Next, let
	\begin{align*}
	\mathcal{P}_1&=\Bigl\{\Bigl[\frac{j}{c_\delta},\frac{j+1}{c_\delta}\Bigr): 0\le j\le c_\delta-1\Bigr\},\\
	\mathcal{P}_2&=\Bigl\{
	\Bigl\{x\in \T\colon
	\sum_{i=1}^{q}H_{n_{q+1}^*(x)}^{h_i}(x)
	\in \Bigl[\frac{j}{c_\epsilon},\frac{j+1}{c_\epsilon}\Bigr)\Bigr\}
	\colon 0\le j\le c_\epsilon-1 \Bigr\},\\
	\mathcal{P}_3&=\Bigl\{ \Bigl\{x\in \T:  \sum_{i=q+1}^{k-1}H_{n_k^*(x)}^{h_i}(x)
	\in \Bigl[\frac{j}{c_\epsilon},\frac{j+1}{c_\epsilon}\Bigr) \Bigr\}:0\le j\le c_\epsilon-1\Bigr\}.
	\end{align*}
	Put $\mathcal{P}=\mathcal{P}_1\vee\mathcal{P}_2\vee\mathcal{P}_3$.
	It is clear that $\mathcal{P}$ is a partition of $\mathbb{T}$ and $\#(\mathcal{P})\le c_\delta c_\epsilon^2$.
	
	Fix two points $x,y$ which are in the same atom of $\mathcal{P}$.
	If there exists $m\ge M_{k}$ with $R_\alpha^mx, R_\alpha^my\in E_{q+1,k-1}^c$, then by \eqref{eq-bh1}
	we have for $q+1\leq i\leq k-1$,
	\[
	H_{m}^{h_i}(x)=H_{n_k^*(x)}^{h_i}(x)
	\text{ and } H_{m}^{h_i}(y)=H_{n_k^*(y)}^{h_i}(y).
	\]
	Thus,
	\begin{align*}
	\Bigl\Vert \sum_{i=q+1}^{k-1}(H_{m}^{h_i}(x)-H_{m}^{h_i}(y))\Bigr\Vert
	=\Bigl\Vert \sum_{i=q+1}^{k-1}(H_{n_k^*(x)}^{h_i}(x)-H_{n_k^*(y)}^{h_i}(y))\Bigr\Vert
	\leq \frac{1}{c_\epsilon}\leq \epsilon,
	\end{align*}
	as $x,y$ are in the same atom in $\mathcal{P}_3$.
	
	By \eqref{20-2} for any $z\in\mathbb{T}$,
	\begin{equation*} 
	\frac{1}{M_{q+1}}\#\{0\le i\le M_{q+1}-1:R_\alpha^iz\in E_{1,q}^c\}\ge1-\sum_{i=1}^\infty\frac{\eta}{2^{i}}>\frac{1}{2}.
	\end{equation*}
	If there exists
	$m\ge M_{k}$ with $R_\alpha^mx, R_\alpha^my\in E_{q+1,k-1}^c$,
	then we can find an integer $M\in [m- M_{q+1}, m-1]$ such that
	$R_\alpha^Mx\in E_{1,q}^c$ and $R_\alpha^My\in E_{1,q}^c$.
	Note that $M\geq m- M_{q+1}>M_{q+1}$.
	By \eqref{eq-bh1}, for $1\leq i \leq q$,
	\[
	H_{M}^{h_i}(x)=H_{n_{q+1}^*(x)}^{h_i}(x)
	\text{ and } H_{M}^{h_i}(y)=H_{n_{q+1}^*(y)}^{h_i}(y).
	\]
	Then
	\begin{align*}
	\Bigl\Vert \sum_{i=1}^{q}(H_{M}^{h_i}(x)-H_{M}^{h_i}(y))\Bigr\Vert
	=\Bigl\Vert \sum_{i=1}^{q}(H_{n_{q+1}^*(x)}^{h_i}(x)-H_{n_{q+1}^*(y)}^{h_i}(y))\Bigr\Vert
	\leq \frac{1}{c_\epsilon}\leq \epsilon,
	\end{align*}
	as $x,y$ are in the same atom in $\mathcal{P}_2$.
	
	As $x,y$ are in the same atom in $\mathcal{P}_1$,
	$\Vert R_\alpha^Mx-R_\alpha^My\Vert=
	\Vert x-y\Vert \le \frac{1}{c_\delta}\le \delta(\epsilon)$.
	Note that $m-M\leq M_{q+1}-1$.
	By \eqref{eq-40} we have
	\[
	\Bigl\Vert \sum_{i=1}^{q}H_{m-M}^{h_{i}}(R_\alpha^Mx)-\sum_{i=1}^{q}H_{m-M}^{h_{i}}(R_\alpha^My)\Bigr\Vert < \epsilon.
	\]
	Hence, if there exists  $m\ge M_{k}$ with $R_\alpha^mx, R_\alpha^my\in E_{q+1,k-1}^c$, then we have
	\begin{align*}
	\bigl\Vert H_{m}^{h_{1,k-1}}(x)-H_{m}^{h_{1,k-1}}(y)\bigr\Vert
	&\le\Bigl\Vert \sum_{i=1}^{q}(H_{m}^{h_{i}}(x)-H_{m}^{h_{i}}(y))\Bigr\Vert +
	\Bigl\Vert \sum_{i=q+1}^{k-1}(H_{m}^{h_i}(x)-H_{m}^{h_i}(y))\Bigr\Vert \\
	&\leq\Bigl\Vert \sum_{i=1}^{q}(H_{M}^{h_{i}}(x)-H_{M}^{h_{i}}(y))\Bigr\Vert \\
	& \qquad \qquad +\Bigl\Vert \sum_{i=1}^{q}H_{m-M}^{h_{i}}(R_\alpha^Mx)-\sum_{i=1}^{q}H_{m-M}^{h_{i}}(R_\alpha^My) \Bigl\Vert \\
	&\qquad \qquad +\Bigl\Vert \sum_{i=q+1}^{k-1}(H_{m}^{h_i}(x)-H_{m}^{h_i}(y))\Bigr \Vert \\
	&\le 3\epsilon.
	\end{align*}
Finally,
	\begin{align*}
	\bar{d}_{n}^{h_{1,k-1}}(x,y)
	&=\frac{1}{n}\sum_{j=0}^{n-1}\Vert H_j^{h_{1,k-1}}(x)-H_j^{h_{1,k-1}}(y)\Vert \\
	&\leq\frac{1}{n}\Biggl(\sum_{\scriptstyle M_k\le j\le n-1 \atop \scriptstyle R_\alpha^jx,R_\alpha^jy\in
E_{q+1,k-1}^c}\Vert H_j^{h_{1,k-1}}(x)-H_j^{h_{1,k-1}}(y)\Vert \\
	& \qquad\qquad +\sum_{\scriptstyle M_k\le j\le n-1 \atop \scriptstyle R_\alpha^jx\in E_{q+1,k-1}}1+
\sum_{\scriptstyle M_k\le j\le n-1 \atop \scriptstyle R_\alpha^jy\in E_{q+1,k-1}}1+\sum_{0\le j\le M_k-1}1\Biggr)\\
	&\leq 3\epsilon+\frac{1}{n}\#(\{0\le j\le n-1:R_\alpha^jx\in E_{q+1,k-1}\})\\
	&\qquad\quad +\frac{1}{n-1}\#(\{0\le j\le n-1:R_\alpha^jy\in E_{q+1,k-1}\})+\frac{M_k}{n-1} \\
	&< 6\epsilon,
	\end{align*}
	where the last inequality follows from \eqref{eq:Ejk} and \eqref{eq-34}.
	
	\medskip
	 \noindent \textbf{Step 2}:  \textit{We will construct a finite cover $\mathcal{Q}$ of $\mathbb{T}$ such that
$$\#(\mathcal{Q})\le10c_\epsilon^4\text{ and } \bar{d}_{n}^{h_k}(x,y)\le 4\epsilon$$
		for any $x,y\in Q\in\mathcal{Q}$.}
	
	\medskip
	There are two cases. The first case is $n\le 2c_\epsilon N_k$.
	In this case, we put
	$$Q_0=\mathbb{T}\setminus \biggl(\bigcup_{-2c_\epsilon N_{k}\le i<(2+2c_\epsilon)N_{k}}[i\alpha-\gamma_k,i\alpha+\gamma_k]
	\biggr),$$ and
	$$Q_{r,s}=\bigcup_{ \frac{r N_{k}}{c_\epsilon}\le i<\frac{ (r+1) N_{k}}{\epsilon}}\biggl[i\alpha+\frac{\gamma_{k} s}{c_\epsilon^2},i\alpha+\frac{\gamma_{k} (s+1)}{c_\epsilon^2}\biggr].$$
	Let
	$$\mathcal{Q}=\{Q_0\}\cup\bigl\{Q_{r,s}: -2c_\epsilon^2\le r\le (2+2c_\epsilon)c_\epsilon-1,-c_\epsilon^2\le s\le c_\epsilon^2-1\bigr\}.$$
	It is clear that $\mathcal{Q}$ is a cover of $\mathbb{T}$ and
	$\#(\mathcal{Q})\le 2c_\epsilon^2\cdot  5c_\epsilon^2=\frac{10c_\epsilon^3}{\epsilon}\le10c_\epsilon^4$.
	For $x,y\in Q_0$, one has $\bar{d}_{n}^{h_{k}}(x,y)=0$ by \eqref{eq-first1}.
	
	Now assume that $x,y\in Q_{r,s}$ for some $r$ and $s$.
	There exist  integers $m_1,m_2\in [\frac{ r N_{k}}{c_\epsilon},\frac{(r+1)N_{k}}{c_\epsilon}]$ and $x_1,y_1\in [\frac{\gamma_{k} s}{c_\epsilon^2},\frac{\gamma_{k} (s+1)}{c_\epsilon^2}]$ such that
	$$x=R_\alpha^{m_1}x_1\text{ and }y= R_\alpha^{m_2}y_1.$$
	Without loss of generality, we can assume that $m_1\leq m_2$.
	For any $1\le m\le n$, one has
	\begin{align*}
	\Vert H_m^{h_{k}}(x)-H_m^{h_{k}}(y)\Vert
	&=\Bigl\Vert \sum_{i=0}^{m-1}(h_{k}(R_\alpha^ix)-h_{k}(R_\alpha^iy))\Bigr\Vert \\
	&\le \Bigl\Vert \sum_{i=m_1}^{m_1+m-1}h_{k}(R_\alpha^ix_1)-\sum_{i=m_2}^{m_2+m-1}h_{k}(R_\alpha^iy_1))\Bigr\Vert \\
	&\le\Bigl\Vert \sum_{i=m_1}^{m_1+m-1}h_{k}(R_\alpha^ix_1)-\sum_{i=m_2}^{m_2+m-1}h_{k}(R_\alpha^ix_1)\Bigr\Vert \\
	&\qquad\qquad +\Bigl\Vert \sum_{i=m_2}^{m_2+m-1}(h_{k}(R_\alpha^ix_1)-h_{k}(R_\alpha^iy_1))\Bigr\Vert \\
	&\le\sum_{i=m_1}^{m_2-1}\Vert h_{k}(R_\alpha^ix_1)\Vert +\sum_{i=m_1+m}^{m_2+m-1}\Vert h_{k}(R_\alpha^ix_1)\Vert +m\cdot\frac{\gamma_k}{c_\epsilon^2}\cdot\frac{1}{N_k\gamma_k}\\
	&\le2(m_2-m_1)\frac{1}{N_k}+m\cdot\frac{\gamma_k}{c_\epsilon^2}\cdot\frac{1}{N_k\gamma_k} \qquad\text{by \eqref{eq:hk-bounded} and \eqref{eq:Lip-continuous}}
	\\
	&\le 4\epsilon.
	\end{align*}
	Hence, summing up we obtain
	$$\bar{d}_{n}^{h_{k}}(x,y)\le 4\epsilon, \text{ for }x,y\in Q\in\mathcal{Q}.$$
	
	\medskip
	The second case is $n> 2c_\epsilon N_k$.
	In this case, we put
	$$Q_0=\mathbb{T}\setminus \biggl(\bigcup_{0\le i<2N_{k}}[i\alpha-\gamma_k,i\alpha+\gamma_k]\biggr),$$
	and
	$$Q_{r,s}=\bigcup_{\frac{rN_{k}}{c_\epsilon}\le i<\frac{ (r+1)N_{k}}{c_\epsilon}}\biggl[i\alpha+\frac{\gamma_{k} s}{c_\epsilon^2},i\alpha+\frac{\gamma_{k} (s+1)}{c_\epsilon^2}\biggr].$$
	Let
	$$\mathcal{Q}=\{Q_0\}\bigcup
	\bigl\{Q_{r,s}: 0\le r\le 2c_\epsilon -1,-c_\epsilon^2\le s\le c_\epsilon^2-1\bigr\}.$$
	
	It is clear that $\mathcal{Q}$ is a cover of $\mathbb{T}$ and $\#\mathcal{Q}\le 10c_\epsilon^4$.
	Given $x,y\in Q_0$, by \eqref{fc-1} and \eqref{eq-first1} one has
	$$\#\{0\le m\le n-1: H_{m}^{h_k}(x)\neq 0\}\le 2N_k$$ and $$\#\{0\le m\le n-1: H_{m}^{h_k}(y)\neq 0\}\le 2N_k.$$
	Then by \eqref{eq:hk-bounded}
	$$\bar{d}_{n}^{h_{k}}(x,y)\le\frac{1}{n}\cdot 4N_k\le 2\epsilon.$$
	
	Now assume that $x,y\in Q_{r,s}$ for some $r$ and $s$.
There exist  integers $m_1,m_2\in [\frac{ r N_{k}}{c_\epsilon},\frac{(r+1)N_{k}}{c_\epsilon}]$
and  $x_1,y_1\in [\frac{\gamma_{k} s}{c_\epsilon^2},\frac{\gamma_{k} (s+1)}{c_\epsilon^2}]$
	such that
	$$x=R_\alpha^{m_1}x_1\text{ and }y= R_\alpha^{m_2}y_1.$$
	Without loss of generality, we can assume that $m_1\leq m_2$.
Recall that $2N_{k}<n\le2N_{k+1}$.	By \eqref{fc-1} and \eqref{eq-first1} one has
	\begin{align}\label{eq-first2}
	h_{k}(R_\alpha^ix_1)=h_{k}(R_\alpha^iy_1)=0
	\end{align}
	for any $2N_k<i\le2N_k+ n\le 2N_k+2N_{k+1}$.
	For any $1\le m\le n$, one has
	\begin{align*}
	\Vert H_m^{h_{k}}(x)-H_m^{h_{k}}(y)\Vert
	&= \Bigl\Vert \sum_{i=m_1}^{m_1+m-1}h_{k}(R_\alpha^ix_1)-\sum_{i=m_2}^{m_2+m-1}h_{k}(R_\alpha^iy_1))\Bigr\Vert \\
	&\le\Bigl\Vert \sum_{i=m_1}^{m_1+m-1}h_{k}(R_\alpha^ix_1)-\sum_{i=m_2}^{m_2+m-1}h_{k}(R_\alpha^ix_1)\Bigr\Vert \\
	&\qquad\qquad +\Bigl\Vert \sum_{i=m_2}^{m_2+m-1}(h_{k}(R_\alpha^ix_1)-h_{k}(R_\alpha^iy_1))\Bigr\Vert \\
	&\le\sum_{i=m_1}^{m_2-1}\Vert h_{k}(R_\alpha^ix_1)\Vert +\sum_{i=m_1+m}^{m_2+m-1}\Vert h_{k}(R_\alpha^ix_1)\Vert \\
	&\qquad\qquad + \Bigl\Vert \sum_{\scriptstyle m_2\le i\le m_2+m-1 }(h_{k}(R_\alpha^ix_1)-h_{k}(R_\alpha^iy_1))\Bigr\Vert \\
	&= 2(m_2-m_1)\frac{1}{N_k}+\biggl\Vert \sum_{\scriptstyle m_2\le i\le m_2+m-1}(h_{k}(R_\alpha^ix_1)-h_{k}(R_\alpha^iy_1))\biggr\Vert
	\quad\text{by \eqref{eq:hk-bounded}}
	\\
	&= 2(m_2-m_1)\frac{1}{N_k}+\biggl\Vert \sum_{\scriptstyle m_2\le i\le m_2+m-1  \atop \scriptstyle i\le 2N_k}(h_{k}(R_\alpha^ix_1)-h_{k}(R_\alpha^iy_1))\biggr\Vert
	\quad\text{by \eqref{eq-first2}}
	\\
	&\le2(m_2-m_1)\frac{1}{N_k}+
	2N_k\cdot\frac{\gamma_k}{c_\epsilon^2}\cdot\frac{1}{N_k\gamma_k}
	\quad\text{by \eqref{eq:Lip-continuous}}  \\
	&\le 4\epsilon.
	\end{align*}
	Hence, summing up we get $$\bar{d}_{n}^{h_{k}}(x,y)\le 4\epsilon, \text{ for }x,y\in Q\in \mathcal{Q}.$$

	\medskip
	 \noindent \textbf{Step 3}:
		\textit{We will construct a finite cover $\mathcal{I}$ of $\mathbb{T}$ such that $$\#(\mathcal{I})\le10c_\epsilon^3\text{ and } \bar{d}_{n}^{h_{k+1}}(x,y)\le 4\epsilon$$
		for any $x,y\in I\in\mathcal{I}$.}
	
	\medskip
	Put
	$$I_0=\mathbb{T}\setminus \biggl(\bigcup_{i=-2N_{k+1}}^{2N_{k+1}} [i\alpha-\gamma_{k+1},i\alpha+\gamma_{k+1}]\biggr),$$
	and
	$$I_{r,s}=\bigcup_{\frac{rN_{k+1}}{c_\epsilon}\le i<
		\frac{ (r+1)N_{k+1}}{c_\epsilon}} \Bigl[i\alpha+\frac{\gamma_{k+1} s}{c_\epsilon^2},i\alpha+\frac{\gamma_{k+1} (s+1)}{c_\epsilon^2}\Bigr].$$
	Put $$\I=\{I_0\}\bigcup\bigl\{I_{r,s}: -2c_\epsilon\le r\le  2c_\epsilon-1,-c_\epsilon^2\le s\le c_\epsilon^2-1
	\bigr\}.$$
	It is clear that $\mathcal{I}$ is a cover of $\mathbb{T}$ and $\#(\mathcal{I})\le 10c_\epsilon^3$. Given $x,y\in I_0$, one has $\bar{d}_{n}^{h_{k+1}}(x,y)=0$ by \eqref{eq-first1}.
	
	Now assume that $x,y\in I_{r,s}$ for  some $r$ and $s$.
	There exist integers $m_1,m_2\in [\frac{ r N_{k}}{c_\epsilon},\frac{(r+1)N_{k}}{c_\epsilon}]$ and $x_1,y_1\in[\frac{\gamma_{k+1} s}{c_\epsilon^2},\frac{\gamma_{k+1} (s+1)}{c_\epsilon^2}]$
	such that
	$$x=R_\alpha^{m_1}x_1\text{ and }y= R_\alpha^{m_2}y_1.$$
	Without loss of generality, we can assume that $m_1\leq m_2$.
	For any $1\leq m\le n$, one has
	\begin{align*}
	\Vert H_m^{h_{k+1}}(x)-H_m^{h_{k+1}}(y)\Vert &=
	\Bigl\Vert \sum_{i=0}^{m-1}(h_{k+1}(R_\alpha^ix)-h_{k+1}(R_\alpha^iy))\Bigr\Vert \\
	&\le\Bigl \Vert \sum_{i=m_1}^{m_1+m-1}h_{k+1}(R_\alpha^ix_1)-\sum_{i=m_2}^{m_2+m-1}h_{k+1}(R_\alpha^iy_1))\Bigr\Vert \\
	&\le \Bigl\Vert \sum_{i=m_1}^{m_1+m-1}h_{k+1}(R_\alpha^ix_1)-\sum_{i=m_2}^{m_2+m-1}h_{k+1}(R_\alpha^ix_1) \Bigr\Vert \\
	&\qquad \qquad +\Bigl\Vert \sum_{i=m_2}^{m_2+m-1}(h_{k}(R_\alpha^ix_1)-h_{k}(R_\alpha^iy_1))
	\Bigr\Vert \\
	&\le2(m_2-m_1)\frac{1}{N_{k+1}}+m\cdot\frac{\gamma_{k+1}}{c_\epsilon^2}\cdot\frac{1}{N_{k+1}\gamma_{k+1}} \qquad\text{by \eqref{eq:hk-bounded} and \eqref{eq:Lip-continuous}}
	\\
	&\le 4\epsilon.
	\end{align*}
	Hence, summing up we have
	$$\bar{d}_{n}^{h_{k+1}}(x,y)\le 4\epsilon, \text{ for }x,y\in Q\in \mathcal{I}.$$
	
	\medskip
  \noindent \textbf{Step 4}: \textit{We will construct a finite cover $\mathcal{T}$ of $\mathbb{T}^2$ such that $$\#(\mathcal{T})\le100c_\epsilon^{11}c_\delta\text{ and } \bar{d}_{n}((x_1,y_1),(x_2,y_2))\le 17\epsilon$$
		for any $(x_1,y_1),(x_2,y_2)\in W\in\mathcal{T}$.}
	
	\medskip
	Note that for any$x\in\mathbb{T}$,
	\[
	\Vert h_{k+2,\infty}(x)\Vert \le\sum_{i=k+2}^\infty\frac{1}{N_i}
	\le\frac{2}{N_{k+2}}.
	\]
	For any $x,y\in\mathbb{T}$ and $1\le m\le n$, by \eqref{eq-34} and $2N_{k}<n\le2N_{k+1}$, one has
	\begin{align*}
	\Vert H_m^{h_{k+2,\infty}}(x)-H_m^{h_{k+2,\infty}}(y)\Vert &=
	\Bigl\Vert \sum_{i=0}^{m-1}(h_{k+2,\infty}(R_\alpha^ix)-h_{k+2,\infty}(R_\alpha^iy))\Bigr\Vert \\
	&\le\frac{m}{N_{k+2}}\le\frac{4N_{k+1}}{N_{k+2}}<\epsilon.
	\end{align*}
	Hence,
	\begin{eqnarray}\label{eq-36}
	\bar{d}_{n}^{h_{k+2,\infty}}(x,y)<\epsilon.
	\end{eqnarray}
	
	\medskip
	Finally, let $\mathcal{S}=\{[\frac{j}{c_\epsilon},\frac{j+1}{c_\epsilon})
	\colon j=0,1\cdots,c_\epsilon-1\}$
	and put
	$$\mathcal{T}=(\mathcal{S}\vee\mathcal{P}\vee\mathcal{Q}\vee\I)\times\mathcal{S}.$$
	It is clear that $\mathcal{T}$ is a finite cover of $\mathbb{T}^2$ with $$\#(\mathcal{T})\le c_\epsilon\cdot c_\delta c_\epsilon^2\cdot 10c_\epsilon^4\cdot 10c_\epsilon^3\cdot c_\epsilon=100c_\epsilon^{11}c_\delta=C(\epsilon).$$
	Hence, for $(x_1,y_1), (x_2,y_2)\in W\in \mathcal{T}$,
	by Steps 1, 2, 3 and \eqref{eq-36}, one has
	\begin{align*}
	\bar d_{n}^h(x_1,x_2)&\le \bar d_{n}^{h_{1,k-1}}(x_1,x_2)+\bar d_{n}^{h_{k}}(x_1,x_2)+\bar d_{n}^{h_{k+1}}(x_1,x_2)+\bar d_{n}^{h_{k+2,\infty}}(x_1,x_2)\\
	&< 6\epsilon+4\epsilon+4\epsilon+\epsilon=15\epsilon.
	\end{align*}
	We deduce that
	\begin{align*}
	\bar{d}_n((x_1,y_1), (x_2,y_2))\le \Vert x_1-x_2\Vert +\Vert y_1-y_2\Vert +\bar d_{n}^h(x_1,x_2)< 17\epsilon.
	\end{align*}
	This implies $\overline{\spann}(n,17\epsilon)\le C(\epsilon)$
	for all $n>K(\epsilon)$,
	which ends the proof.
\end{proof}

\subsection{The construction of the system in Proposition~\ref{prop:key-example2}}
First we need the following Furstenberg's dichotomy result.
\begin{prop}[\cite{Fur61}]
	Suppose $(\Omega_0,\mu_0,T_0)$ is a uniquely ergodic topological dynamical system with $\mu_0$ being the unique ergodic measure,
	and $h:\Omega_0\to\mathbb{T}$ is a continuous function.
	Let $T :\Omega_0\times\mathbb{T}$ be defined by $T(x,y)=(T_0(x),y+h(x))$.
	Then exactly one of the following is true:
	\begin{itemize}
		\item[(1)]  $T$ is uniquely ergodic and
		$\mu_0\times m_{\mathbb{T}}$  is the unique invariant measure;
		\item[(2)]  there exists a measurable map
		$g : \Omega_0\to\mathbb{T}$
		and a non-zero integer $s$
		such that $s\cdot h(x) = g(T_0(x))-g(x)$
		for $\mu_0$-almost every $x\in \Omega_0$.
	\end{itemize}
\end{prop}

Now we modify the example $(\mathbb{T}^2,T)$ in the previous subsection
to be uniquely ergodic.
As $(\mathbb{T}^2,T)$ is not uniquely ergodic,
by Furstenberg's dichotomy result there is an $m_\mathbb{T}$-measurable function $g(x)$ and a non-zero integer $s$ such that
\begin{equation}\label{11-14-2}
	s\cdot h(x)=g(x+\alpha)-g(x)
\end{equation}
	for $m_\mathbb{T}$-a.e. $x\in\mathbb{T}$.
	We define
	\[\phi\colon
	\mathbb{T}^2\to\mathbb{T}^2,\quad (x,y)\mapsto(x,s\cdot y)\]
	and
	\[\widetilde{T}\colon\mathbb{T}^2\to\mathbb{T}^2,
	\quad (x,y)\mapsto(x+\alpha,y+s\cdot h(x)).\]
	Then $\widetilde{T}\circ \phi=\phi\circ T$, in other words, the following diagram commutes.
	\[\xymatrix{
		\mathbb{T}^2 \ar[r]^{T}\ar[d]_{\phi} &    \mathbb{T}^2 \ar[d]_{\phi}\\
		\mathbb{T}^2 \ar[r]^{\widetilde{T}} &    \mathbb{T}^2
	}
	\]
	Take an irrational number $\beta\in\mathbb{R}$ such that
	$\alpha$ and $\beta$ are rationally independent.
	Then the system defined by
	\[T_{\alpha,\beta}\colon  \mathbb{T}^2\to\mathbb{T}^2,\quad  (x,y)\to(x+\alpha,y+\beta)\] is uniquely ergodic and $m_{\mathbb{T}^2}$ is the unique invariant measure.
	Finally, we define
	\[\widetilde{T}_{\beta}\colon  \mathbb{T}^2\to\mathbb{T}^2,
	\quad (x,y)\mapsto(x+\alpha,y+s\cdot h(x)+\beta).\]
	We will show that the system $(\mathbb{T}^2,\widetilde{T}_{\beta})$ is the one we need.
	It is clear that $(\mathbb{T}^2,\widetilde{T}_{\beta})$ is distal.
	
	\medskip
\begin{prop}
$(\mathbb{T}^2,\widetilde{T}_{\beta})$ is uniquely ergodic and minimal.
\end{prop}	
\begin{proof}
Let $K=\{x\in \mathbb{T}\colon s\cdot h(x)=g(x+\alpha)-g(x)\}$ and $\pi\colon \mathbb{T}^2\to\mathbb{T}^2$, $(x,y)\mapsto(x,y-g(x))$.
By \eqref{11-14-2} one has $m_\mathbb{T}(K)=1$.
It is easy to see that
	$\pi\colon K\times\mathbb{T}\to K\times\mathbb{T}$
is an invertible map with $\pi\circ\widetilde{T}_{\beta}=T_{\alpha,\beta}\circ\pi$.
For each $\widetilde{T}_{\beta}$-invariant measure $\mu$, we have $\mu(K\times\mathbb{T})=1$ and $\mu\circ\pi^{-1}$ is
$T_{\alpha,\beta}$-invariant. We have $\mu\circ\pi^{-1}=m_{\mathbb{T}^2}$ since $m_{\mathbb{T}^2}$ is the unique invariant probability measure of $T_{\alpha,\beta}$.
Thus, $\mu=m_{\mathbb{T}^2}\circ\pi$.
This implies that 	$m_{\mathbb{T}^2}\circ\pi$ is the only invariant measure for
$(\mathbb{T}^2,\widetilde{T}_{\beta})$. Moreover, $(\mathbb{T}^2,\widetilde{T}_{\beta})$ is minimal since the only invariant measure $m_{\mathbb{T}^2}\circ \pi$ is of full support.
\end{proof}	

\begin{prop}
$(\mathbb{T}^2,\widetilde{T}_{\beta})$ is not equicontinuous.
\end{prop}
\begin{proof}
	It is sufficient to show for any $\epsilon>0$,
	there exist $(x_1,y_1), (x_2,y_2)\in\mathbb{T}^2$ and a positive integer $n$ such that $d((x_1,y_1),(x_2,y_2))\le\epsilon$ and
$d(\widetilde{T}_{\beta}^n(x_1,y_1),\widetilde{T}_{\beta}^n(x_2,y_2))\ge\frac{1}{200}$.
	
	Assuming that $10^p\le |s|<10^{p+1}$ for some non-negative integer $p$. Given $\epsilon>0$, there exists $k\in\mathbb{N}$
such that $k>p+10$ and $l_k+\delta_k<\epsilon.$ Put $x'=\delta_k+\frac{1}{2}l_k.$
	One has $R_\alpha^{i}x'\in E_k^c$ and $h_k(R_\alpha^{i}x')=0$ for $i=0,1,2,\cdots,N_k-1$.
	By \eqref{20-2}, for any $x\in \mathbb{T}$,
	\begin{equation*}
	\frac{1}{M_k}\#\{0\le i\le M_k-1:R_\alpha^ix \in E_{1,k-1}^c\}\ge1-\sum_{i=1}^\infty\frac{\eta}{2^{i}}>\frac{1}{2}.
	\end{equation*}
	Then there are integers $n_1\in [0,M_k-1]$ and $n_2\in [10^{k-p-2}M_k-M_k,10^{k-p-2}M_k-1]$ such that
	$R_\alpha^{n_1}0,R_\alpha^{n_1}x', R_\alpha^{n_2}0,R_\alpha^{n_2}x'\in E_{1,k-1}^c$.
	By using Lemma \ref{lm} and the fact $R_\alpha^{n_1}x',R_\alpha^{n_2}x'\in E_k^c$, we have
	\begin{align*}
	H_{n_2-n_1}^{h}(R_\alpha^{n_1}0)-H_{n_2-n_1}^{h}(R_\alpha^{n_1}x')
	&=H_{n_2-n_1}^{h_{1,k-1}}(R_\alpha^{n_1}0)-H_{n_2-n_1}^{h_{1,k-1}}(R_\alpha^{n_1}x')
	+H_{n_2-n_1}^{h_k}(R_\alpha^{n_1}0)\\
	&\qquad -H_{n_2-n_1}^{h_k}(R_\alpha^{n_1}x')+H_{n_2-n_1}^{h_{k+1,\infty}}(R_\alpha^{n_1}0)-H_{n_2-n_1}^{h_{k+1,\infty}}(R_\alpha^{n_1}x')\\
	&= H_{n_2-n_1}^{h_k}(R_\alpha^{n_1}0)+H_{n_2-n_1}^{h_{k+1,\infty}}(R_\alpha^{n_1}0)-H_{n_2-n_1}^{h_{k+1,\infty}}(R_\alpha^{n_1}x')\\
	&=(n_2-n_1)\frac{1}{N_k}+H_{n_2-n_1}^{h_{k+1,\infty}}(R_\alpha^{n_1}0)-H_{n_2-n_1}^{h_{k+1,\infty}}(R_\alpha^{n_1}x').
	\end{align*}
	Moreover, we have
	\begin{align*}
	\Vert H_{n_2-n_1}^{h_{k+1,\infty}}(R_\alpha^{n_1}0)-H_{n_2-n_1}^{h_{k+1,\infty}}(R_\alpha^{n_1}x') \Vert
	&\le \sum_{i=k+1}^\infty(|H_{n_2-n_1}^{h_i}(R_\alpha^{n_1}0)|+|H_{n_2-n_1}^{h_i}(R_\alpha^{n_1}x')|)\\
	&\le \sum_{i=k+1}^\infty 2(n_2-n_1)\frac{1}{10^iM_i}\le \sum_{i=k+1}^\infty \frac{2\cdot 10^{k-p-2}M_k}{10^iM_i}\\
	&\le2\sum_{i=k+1}^\infty\frac{10^{-p-2}}{10^{i-k}}
	=\frac{2}{9}\cdot 10^{-p-2}.
	\end{align*}
	Note that $10^{k-p-2}M_k-2M_k\le n_2-n_1\le 10^{k-p-2}M_k$.
	One has
	$$|s|\cdot \Vert H_{n_2-n_1}^{h}(R_\alpha^{n_1}0)-H_{n_2-n_1}^{h}(R_\alpha^{n_1}x')\Vert
	\le |s|\cdot \Bigl (10^{-p-2}+\frac{2}{9}\cdot 10^{-p-2}\Bigr) \le\frac{2}{10}.$$
	and
	$$|s|\cdot \Vert H_{n_2-n_1}^{h}(R_\alpha^{n_1}0)-H_{n_2-n_1}^{h}(R_\alpha^{n_1}x')\Vert
	\ge|s|\cdot\Bigl(10^{-p-2}-\frac{2}{10^k}-\frac{2}{9}\cdot 10^{-p-2}\Bigr)\ge\frac{1}{200}.$$
	Let $x_1=R_\alpha^{n_1}0$, $x_2=R_\alpha^{n_1}x'$,
	$y_1=y_2=0$ and $n=n_1-n_2$.
	Then
	\[
	d((x_1,y_1),(x_2,y_2))=
	\Vert R_\alpha^{n_1}0-R_\alpha^{n_1}x' \Vert=
	\Vert x' \Vert =\delta_k+\frac{1}{2}\ell_k<\epsilon.
	\]
	and
	\begin{align*}
	d(\widetilde{T}_\beta^{n}(x_1,y_1),&
	\widetilde{T}_\beta^{n}(x_2,y_2))\\
	&=d\Bigl((R_\alpha^{n_2}0,s\cdot H^h_{n_2-n_1}(R_\alpha^{n_1}0)+(n_2-n_1)\beta),
	\\ &\qquad\qquad (R_\alpha^{n_2}x',s\cdot H^h_{n_2-n_1}(R_\alpha^{n_1}x')+(n_2-n_1)\beta\Bigr)\\
	&\ge \bigl\Vert s\cdot (H_{n_2-n_1}^{h}(R_\alpha^{n_1}0)-H_{n_2-n_1}^{h}(R_\alpha^{n_1}x'))\bigr\Vert \ge\frac{1}{200}.
	\end{align*}
	This implies that $(\mathbb{T}^2,\widetilde{T}_\beta)$ is not equicontinuous.
\end{proof}
	
\begin{prop}
		$(\mathbb{T}^2,\widetilde{T}_{\beta})$ has bounded topological complexity with respect to $\{\bar{d}_n\}$.
\end{prop}
\begin{proof}
	For $\epsilon>0$, let $\mathscr{T}$, $c_\epsilon$ and $c_\delta$ be defined in Proposition \ref{pro-1}. Then for $(x_1,y_1),(x_2,y_2)\in W\in \mathscr{T}$, one has
	\begin{align*}
	\bar d_{n}^{s\cdot h+\beta}(x_1,x_2)&=\frac{1}{n}\sum_{m=0}^{n-1}
	\Vert H_m^{sh+\beta}(x)-H_m^{sh+\beta}(y)\Vert\\ &=\frac{1}{n}\sum_{m=0}^{n-1}\Vert sH_m^{h}(x)-sH_m^{h}(y)\Vert \\
	&\le\frac{1}{n}\sum_{m=0}^{n-1}|s|\cdot \Vert H_m^{h}(x)-H_m^{h}(y)\Vert \\
	&\le |s|\cdot \bar d_{n}^{h}(x_1,x_2)\le15|s|\epsilon.
	\end{align*}
	and
	\begin{align*}
	\bar{d}_n((x_1,y_1), (x_2,y_2))\le \Vert x_1-x_2\Vert +\Vert y_1-y_2\Vert +\bar d_{n}^h(x_1,x_2)\le (15|s|+2)\epsilon.
	\end{align*}
	Hence $\overline{\spann}(n,(15|s|+2)\epsilon)\le100c_\epsilon^{11}c_\delta$. Thus $(\mathbb{T}^2,\widetilde{T}_{\beta})$
has bounded topological complexity with respect to $\{\bar{d}_n\}$.
\end{proof}

\section{An Example by Cyr and Kra}
We first introduce some concepts.
Following \cite{D06}, by an \textit{assignment}, 
we mean a function $\Psi$ defined on an abstract metrizable Choquet simplex $\mathcal{P}$, whose ``value'' are measure-theoretic dynamical systems, i.e., for $p\in\mathcal{P}$, $\Psi(p)$ has the form $(X_p,\mathcal{B}_p,\mu_p,T_p)$.
Two assignments, $\Psi$ on a simplex $\mathcal{P}$ and $\Psi'$ on a simplex $\mathcal{P}'$,
are said to be \emph{equivalent} if there exists an affine homeomorphism
$\pi\colon\mathcal{P}\to\mathcal{P}'$ of Choquet simplexes such that for every $p\in\mathcal{P}$
the systems $\Psi(p)$ and $\Psi'(\pi(p))$ are isomorphic as measure-theoretic dynamical systems.
A topological dynamical system $(X,T)$ determines an assignment on the simplex of $T$-invariant probability measures by the rule $\mu\to (X,\mathcal{B}_X,\mu,T)$, where $\mathcal{B}_X$ denotes the collection of Borel sets on $X$.
By \cite[Theorem 1]{D06} or \cite[Theorem 1]{KO06},
we know that if $Y$ is zero dimensional and
$(Y,S)$ has no periodic points, then the assignment
determined by $(Y,S)$ is equivalent to an assignment determined by some
minimal system $(X,T)$.
If $(Y,S)$ is invertible, then we can require that $(X,T)$ is also invertible \cite {KO06}.
Applying \cite[Theorem 1]{D06} or \cite[Theorem 1]{KO06},
there is a minimal system $ (X,T)$
whose assignment is equivalent to that of $(Y,S)$.

\begin{prop}\label{prop:Cyr-Kra}
There exists a minimal system with bounded
complexity with respect to $\{\bar{d}_n\}$ for an invariant measure $\mu$, for which there exist two non-isomorphic ergodic measures
in the ergodic decomposition.
\end{prop}
\begin{proof}
Pick two Sturmian shifts $(Y_1,\sigma)$  and $(Y_2,\sigma)$ in the full shifts
$(\{0,1\}^{\mathbb{Z}},\sigma)$  and $(\{2,3\}^{\mathbb{Z}},\sigma)$
respectively.
Then $(Y_1,\sigma)$  and $(Y_2,\sigma)$ are minimal and uniquely ergodic.
Let $\nu_1$ and $\nu_2$ be the unique invariant measure
of  $(Y_1,\sigma)$  and $(Y_2,\sigma)$ respectively.
Then both $\nu_1$ and $\nu_2$ have discrete spectrum.
We can require that the spectra of $\nu_1$ and $\nu_2$ are different and then $\nu_1$ and $\nu_2$ are not isomorphic.
Let $Y=Y_1\cup Y_2\subset \{0,1,2,3\}^{\mathbb{Z}}$.
It is clear that $Y$ is zero dimensional and
 $(Y,\sigma)$ has no periodic points.

By \cite[Theorem 1]{D06} or \cite[Theorem 1]{KO06},
there is a minimal system $(Y,S)$
 whose assignment is equivalent to that of $(Y,\sigma)$.
This means that $(X,T)$ carries exactly two ergodic measures, $\mu_1$ and $\mu_2$,
and $(X,\mathcal{B}_X,\mu_i,T)$ is isomorphic to $(Y_i,\mathcal{B}_{Y_i},\nu_i,\sigma)$.
Let $\mu=\frac{1}{2}\mu_1+\frac{1}{2}\mu_2$.
As both $\mu_1$ and $\mu_2$ have discrete spectrum, so is $\mu$.
By Proposition~\ref{prop:discrete-specturm-HWY}
$\mu$ has bounded  complexity
with respect to  $\{\bar{d}_n\}$.
But the ergodic measures in the ergodic decomposition of $\mu$ are
$\mu_1$ and $\mu_2$, which are not isomorphic.
\end{proof}

\begin{rem}
It should be noticed that the same idea of construction in Proposition~\ref{prop:Cyr-Kra} can be used to provide countably many non-isomorphic ergodic
measures in the ergodic decomposition, but the uncountably many case is still not clear.	
\end{rem}

\end{document}